\documentclass[a4paper,11pt,reqno]{amsart}
\usepackage{latexsym,amsmath,amssymb,amsfonts}
\usepackage{a4wide}
\usepackage{graphicx}
\usepackage{hyperref}

\usepackage{color}

\newcommand{\xupref}[2]{\hspace{-0.3ex}\stackrel{\eqref{#1}}{#2}} 
\newcommand{\upupref}[3]{\hspace{-3ex}\stackrel{\eqref{#1},\eqref{#2}}{#3}}

\theoremstyle{plain}
\begingroup
\newtheorem{theorem}{Theorem}[section]
\newtheorem{lemma}[theorem]{Lemma}
\newtheorem{proposition}[theorem]{Proposition}
\newtheorem{corollary}[theorem]{Corollary}
\endgroup
\theoremstyle{definition}
\begingroup
\newtheorem{definition}[theorem]{Definition}
\endgroup
\theoremstyle{remark}
\newtheorem{remark}[theorem]{Remark}

\numberwithin{equation}{section}

\newcommand{\e}{\varepsilon}
\newcommand{\vphi}{\varphi}
\newcommand{\N}{\mathbb N}
\newcommand{\Z}{\mathbb Z}
\newcommand{\R}{\mathbb R}

\newcommand{\de}{\,\mathrm{d}}

\newcommand{\supp}{{\rm supp\,}}

\renewcommand{\setminus}{\backslash}

\newcommand{\meas}{\mathcal{M}^+([0,\infty))}
\newcommand{\inte}{\int_{[0,\infty)}}
\newcommand{\lambar}{\bar{\lambda}_\rho}
\newcommand{\Cbar}{\bar{C}}
\newcommand{\phibar}{\bar{\vphi}}

\title[Self-similar solutions to coagulation equations with time-dependent tails]{Self-similar solutions to coagulation equations with time-dependent tails: the case of homogeneity one}
\author{Marco Bonacini \and Barbara Niethammer \and Juan J. L.\ Vel\'{a}zquez}
\date{\today}
\address{Institut f\"ur Angewandte Mathematik, Universit\"at Bonn, Endenicher Allee 60, 53115 Bonn, Germany}
\email{bonacini@iam.uni-bonn.de; niethammer@iam.uni-bonn.de; velazquez@iam.uni-bonn.de}


\begin{document}

\begin{abstract}
	We prove the existence of a one-parameter family of self-similar solutions with time dependent tails for Smoluchowski's coagulation equation, for a class of kernels $K(x,y)$ which are homogeneous of degree one and satisfy $K(x,1)\to k_0>0$ as $x\to 0$. In particular, we establish the existence of a critical $\rho_*>0$ with the property that for all $\rho\in(0,\rho_*)$ there is a positive and differentiable self-similar solution with finite mass $M$ and decay $A(t)x^{-(2+\rho)}$ as $x\to\infty$, with $A(t)=e^{M(1+\rho)t}$. Furthermore, we show that (weak) self-similar solutions in the class of positive measures cannot exist for large values of the parameter $\rho$.
\end{abstract}

\maketitle


\section{Introduction} \label{sect:intro}

Smoluchowski's coagulation equation \cite{Smo27} is a classical mean-field model to describe irreversible aggregation of clusters through binary collisions. The evolution of the number density $f(\xi,t)$ of clusters with mass $\xi$ at time $t$ is governed by the equation
\begin{equation} \label{smol}
\partial_t f(\xi,t) = \frac12\int_0^\xi K(\xi-\eta,\eta) f(\xi-\eta,t)f(\eta,t)\de \eta - f(\xi,t)\int_0^\infty K(\xi,\eta)f(\eta,t)\de \eta\,,
\end{equation}
where the kernel $K(\xi,\eta)$ prescribes the rate at which clusters of size $\xi$ and $\eta$ coagulate. The first term on the right-hand side of \eqref{smol} describes the formation of particles of size $\xi$ due to the merging of two particles of mass $\eta$ and $\xi-\eta$ respectively, while the second term takes into account that particles of size $\xi$ can combine with particles of any other size.

For homogeneous kernels, a central question in the qualitative analysis of \eqref{smol} is the so-called \emph{scaling hypothesis}, which predicts that the long-time behaviour of solutions to \eqref{smol} is universal and captured by \emph{self-similar solutions}. This issue is by now well understood for the solvable kernels \cite{CMM10,MP04}: the constant kernel $K(\xi,\eta)=2$, the additive kernel $K(\xi,\eta)=\xi+\eta$, and the multiplicative kernel $K(\xi,\eta)=\xi\eta$; for such kernels, solutions can be computed explicitly via Laplace transform. Rigorous existence and regularity results for self-similar solutions, both with finite mass and with fat tails respectively, have also been established for nonsolvable kernels with homogeneity strictly smaller than 1, see \cite{CanMisch11,EM06,EMR05,FL05,FL06,NTV16,NV13}. In all these results the  additional assumption $K(\xi,1)=O(\xi^{-a})$ as $\xi\to0$, with $a<1-\gamma$, is needed. One of the major open problems that remains is the uniqueness of self-similar profiles with given decay behaviour. Only recently some rigorous results have been obtained for particular cases in \cite{Lau18,LNV18,NTV16b,Thr}.

It is the purpose of this paper to investigate the existence of self-similar solutions to \eqref{smol} in the case of homogeneous kernels of degree one, for which, apart from the solvable additive kernel, no rigorous results have been obtained so far. It is worth to notice that such class of kernels represents a borderline case: indeed, it is well known that the total mass is conserved along the evolution for kernels with homogeneity $\gamma\leq1$, while if $K$ grows too fast at infinity, for instance if $K$ is homogeneous of degree strictly larger than one, solutions exhibit the phenomenon of \emph{gelation}, that is, roughly speaking, infinite large clusters are created at finite time and the total mass decreases.

In the case of homogeneity one, as already noticed in \cite{vDE88}, the picture is different depending on the behaviour of $K(\xi,1)$ as $\xi\to0$. In particular, we shall distinguish between class-I kernels, for which $K(\xi,1)\to0$ as $\xi\to0$, and class-II kernels, for which $K(\xi,1)\to k_0>0$ as $\xi\to0$. In this paper we will only deal with kernels of class-II; one example of such kernels is $K(\xi,\eta)=(\xi^\frac{1}{3}+\eta^{\frac13})^3$, which has been formally derived for particles moving in a shear flow (see also \cite{NNTV,NV} for the derivation and analysis of the corresponding linear version of the equation). We instead refer to \cite{HNV16} for a discussion with numerical simulations of the long-time behaviour of solutions to \eqref{smol} for class-I kernels.

The prototype of class-II kernels is of course the additive kernel $K(\xi,\eta)=\xi+\eta$, for which it has been proved in \cite{MP04} that there exists a one-parameter family of self-similar solutions with finite mass, one of which having exponential decay at infinity and finite second moment, the others decaying like power laws and being characterized by the different divergence behaviours of the second moment. It has been conjectured in \cite{HNV16} that a similar result should hold also for nonsolvable class-II kernels, and it is the purpose of this paper to provide a rigorous answer to this question.

In particular, we will establish in Theorem~\ref{thm:exist} the existence of a family of nonnegative self-similar solutions with finite mass, depending on a positive parameter $\rho$ smaller than a critical value $\rho_*>0$. Such solutions are characterized in terms of their asymptotic decay. A remarkable novelty with respect to previous results is that in the present case self-similar solutions exhibit \emph{time-dependent tails}, see \eqref{time-dep} below; such behaviour is truly different from that of the fat-tail solutions obtained so far for kernels with homogeneity $\gamma<1$, as will be explained below. In a second paper \cite{BNV} we show that such  self-similar solutions with  tail-dependent tails also exist for a class of kernels that are homogeneous of degree $\gamma\in(-\infty,1)$ when $K(\xi,1)\sim \xi^{\gamma-1}$ as $\xi\to0$. The analysis relies on the same methods introduced here, but is significantly more involved due to the presence of a sharp variation of the solution in a small transition layer, which poses additional technical challenges.

Finally, in our second main result we show that (weak) self-similar solutions cannot exist for large values of the parameter $\rho$ (see Theorem~\ref{thm:nonexist}).

\subsection*{Self-similar solutions}
In what follows, it is convenient to reformulate \eqref{smol} in a conservative form:
\begin{equation}\label{smol2}
\partial_t(\xi f(\xi,t)) = - \partial_\xi \biggl( \int_0^\xi\int_{\xi-\eta}^{\infty} K(\eta,\zeta)\eta f(\eta,t)f(\zeta,t)\de \zeta \de \eta \biggr) =: - \partial_\xi \bigl(J[f](\xi,t)\bigr)\,.
\end{equation}
For kernels homogeneous of degree one, the mass density function of a self-similar solution to \eqref{smol2} has the form
\begin{equation} \label{selfsim1}
\xi f(\xi,t) = e^{-bt} g(\xi e^{-bt}), \qquad b>0,
\end{equation}
where the self-similar profile $g$ solves
\begin{equation} \label{selfsim3}
b\partial_x(xg(x)) = \partial_x \biggl( \int_0^x \int_{x-y}^\infty \frac{K(y,z)}{z}g(y)g(z)\de z \de y \biggr)\,.
\end{equation}
Integrating in $x$, we then look for nonnegative solutions $g$ to
\begin{equation} \label{selfsim4}
b x g(x) = \int_0^x\int_{x-y}^\infty \frac{K(y,z)}{z}g(y)g(z)\de z \de y\,,
\end{equation}
with finite mass
\begin{equation}\label{mass}
\int_0^\infty g(x)\de x = M\,.
\end{equation}
Formal asymptotics suggest (see Section~\ref{sect:euristics}) that, if solutions to \eqref{selfsim4} for a given $b>0$ exist, and if they have some power law behaviour $g(x)\sim\frac{1}{x^{1+\rho}}$ as $x\to\infty$, then the relation between the exponent $\rho$ and $b$ is
\begin{equation} \label{b}
b= M\biggl(\frac{1+\rho}{\rho}\biggr)\,.
\end{equation}
The behaviour of $g$ at the origin is in this case $g(x)\sim x^{-\frac{1}{1+\rho}}$.
There is hence a one-to-one correspondence between $b$ and the exponent $\rho$ characterizing the decay behaviour of the solutions.

As already remarked, for the class of kernels considered here self-similar solutions have time-dependent tails. Indeed,  assume to have a self-similar profile $g$ solving \eqref{selfsim4} which decays like $x^{-(1+\rho)}$ as $x\to\infty$. Then it follows from \eqref{selfsim1} that the number density $f$ of the corresponding self-similar solution has a time-dependent tail of the form
\begin{equation} \label{time-dep}
f(\xi,t) \sim \frac{A(t)}{\xi^{2+\rho}}, \qquad\text{with }A(t)=e^{M(1+\rho)t}\,.
\end{equation}
The time-dependence of the coefficient $A(t)$ is a remarkable property of the solutions constructed in this paper, in contrast with those obtained before in \cite{NTV16,NV13} for kernels satisfying the assumptions
\begin{equation} \label{gammasmall}
\gamma<1 \qquad\text{and}\qquad K(\xi,1)=O(\xi^{-a}) \quad\text{as }\xi\to0,\text{ with }a<1-\gamma
\end{equation}	
(where $\gamma$ is the homogeneity of the kernel). It is worth to comment briefly on the differences between these two classes of solutions. The main idea in the construction is to look for solutions which behave like
\begin{equation} \label{powerlaw}
f(\xi,t)\sim A(t)\xi^{-\alpha}
\end{equation}
for large values of $\xi$, for some $\alpha>1$. Denoting by $Q[f]$ the right-hand side in the equation \eqref{smol}, it turns out that under the assumptions \eqref{gammasmall}--\eqref{powerlaw} one has $Q[f]=O(\xi^{-\beta})$, for some $\beta>\alpha$. Then the coagulation equation \eqref{smol} becomes for large values
$$
\bigl(\partial_t A(t)\bigr)\xi^{-\alpha} = O(\xi^{-\beta})\,, \qquad \xi>>1,
$$
and, since $\beta>\alpha$, we obtain $\partial_tA(t)=0$. The tail might therefore be expected to be stationary. Notice in particular that the solutions with fat tails constructed in \cite{NTV16,NV13} solve at $\xi\to\infty$ approximately the equation $\partial_tf=0$. The reason why, among the possible behaviours of solutions to this equation, we restrict to power laws as in \eqref{powerlaw}, is that they are the only functions which can be written in the self-similar form $f(\xi,t)=\lambda(t)\Phi(\xi/\mu(t))$ with $\mu(t)\to\infty$ as $t\to\infty$.

We emphasize that the previous heuristic argument relies on the fact that, for kernels satisfying \eqref{gammasmall}, the coagulation term $Q[f]$ does not contribute to the asymptotics of the solution for large values of $\xi$. In the case considered in this paper, however, that is $\gamma=1$ and $K(0,1)=1$, there is a nontrivial contribution of the coagulation kernel that yields the time-dependent character of the tails highlighted in \eqref{time-dep}. Indeed, by considering power-law solutions as in \eqref{time-dep}, it is possible to approximate the term $J[f]$ on the right-hand side of \eqref{smol2}, namely
$$
J[f](\xi,t) \sim \Bigl(\frac{\rho+1}{\rho}\Bigr)A(t)M\xi^{-\rho} \qquad\text{as }\xi\to\infty
$$
(the details of this computation can be seen in Section~\ref{sect:euristics}).
Then \eqref{smol2} becomes for large values of $\xi$
\begin{equation}
\bigl(\partial_tA(t)\bigr) \xi^{-(1+\rho)} - MA(t)(\rho+1) \xi^{-(1+\rho)} = 0\,,
\end{equation}
whence the second formula in \eqref{time-dep} follows. A similar argument indicates that it is possible to have similar solutions with time-dependent tails also in the case of homogeneity $\gamma<1$ if the kernel behaves as $K(\xi,1)\sim \xi^{\gamma-1}$ as $\xi\to0$, and we prove this rigorously in \cite{BNV}.

\subsection*{Main results}
In order to state rigorously our results, we now formulate the precise assumptions on the rate kernel: $K$ is a continuous, nonnegative and symmetric map
\begin{equation} \label{kernel1}
K\in C([0,\infty)\times[0,\infty)), \qquad
K(x,y) = K(y,x) \geq0 \quad\text{for all }x,y\in[0,\infty),
\end{equation}
homogeneous of degree one
\begin{equation} \label{kernel2}
K(a x,a y) = a K(x,y) \qquad\text{for all }x,y\in(0,\infty),\, a>0,
\end{equation}
and such that for some constants $K_0>0$ and $\alpha >0$
\begin{equation} \label{kernel3}
|K(x,1)-1| \leq K_0 x^\alpha
\qquad\text{for every }x\in[0,1].
\end{equation}
In our first result we establish the existence of a one-parameter family of nonnegative self-similar solutions with finite mass to \eqref{selfsim4}, and we explicitly determine their asymptotic decay, for the class of kernels satisfying the previous assumptions.

\begin{theorem} \label{thm:exist}
Assume that the kernel $K$ satisfies assumptions \eqref{kernel1}, \eqref{kernel2} and \eqref{kernel3}.
Then there exists $\rho_*>0$ such that for every $\rho\in(0,\rho_*)$ and for $b=\frac{1+\rho}{\rho}$ there is a positive solution $g\in C^1((0,\infty))$ to \eqref{selfsim4} with unit mass and satisfying
\begin{equation} \label{asymp}
g(x) \sim x^{-\frac{1}{1+\rho}}\quad\text{as }x\to0, \qquad
g(x) \sim \frac{1}{x^{1+\rho}} \quad\text{as }x\to\infty.
\end{equation}
\end{theorem}

In the statement we have considered without loss of generality the case of unit mass $M=1$: it obviously follows that for every $M>0$ and $\rho\in(0,\rho_*)$ there is a solution $g$ satisfying \eqref{mass} and \eqref{asymp}, for $b$ as in \eqref{b}. Notice that this result resembles the situation for the additive kernel, for which a critical exponent  $\rho_{crit}=1$ exists, above which there exists no nonnegative solution;  our result does not give a critical value $\rho_{crit}$,  we however expect that $\rho_{crit}$ is in general not equal to one, but depends on the details of the kernel (see \cite{HNV16} for a justification of this conjecture).
In fact, we conjecture that there are cases where $\rho_{crit}=\infty$ as explained in Remark~\ref{rm:beta0} below.

The proof of the theorem is given in Section~\ref{sect:ex1} and Section~\ref{sect:ex2}. The strategy is based on a linearization around the explicit solution of an approximate equation, and on a fixed point argument. This provides the existence of a continuous solution satisfying the first condition in \eqref{asymp} (Proposition~\ref{prop:fixedpoint} and Corollary~\ref{cor:fixedpoint}); however, in order to apply the contraction mapping principle, we have to work in a space of functions with a non-optimal decay at infinity, and due to this reason we are not able to directly obtain also the second condition in \eqref{asymp}.
The exact asymptotic of the solution at infinity will be recovered in a second step, together with its strict positivity: to do so, we prove the differentiability of the solution and a decay estimate on its derivative in the form $|g'(x)|\leq\frac{C}{x}|g(x)|$ (Lemma~\ref{lem:continuation}), which allows us to obtain a good approximation of the equation and to complete the proof.

Since our method relies on a contraction principle argument, we also obtain uniqueness of the solution within the class of functions in which we apply the fixed point (see \ref{spacefixedpoint}). However, the space contains an artificial smallness condition, and a general uniqueness statement seems to require different techniques.

\medskip
Under an additional assumption on the kernel we complement Theorem~\ref{thm:exist} with a corresponding nonexistence result: more precisely, 
we replace \eqref{kernel3} by the condition
\begin{equation} \label{kernel4}
\lim_{\xi\to0} \frac{K(1,\xi) - 1}{\xi^\alpha} = \beta_0
\end{equation}
for some $\beta_0>0$ and $\alpha\in(0,1)$.
Furthermore, we assume strict positivity of the kernel:
\begin{equation} \label{kernel5}
K(1,\xi) \geq k_0 >0 \qquad\text{for all }\xi\geq0\,.
\end{equation}
Then we show that for sufficiently large $\rho$, given $M>0$ and $b$ satisfying \eqref{b}, equation \eqref{selfsim4} does not have any solution with mass $M$ in the class of positive Radon measures $\meas$ which
are not supported in the origin. By saying that a measure $g\in\meas$ is a solution of \eqref{selfsim4} we mean that for every test function $\theta\in C([0,\infty))$ with compact support one has
\begin{equation*}
b \inte x\theta(x)g(x)\de x
= \inte \de y \inte \de z \frac{K(y,z)}{z}  g(y)g(z) \int_y^{y+z}\theta(x)\de x\,.
\end{equation*}
Here and in the following, with some abuse of notation, we denote by $\int_A\theta(x)g(x)\de x$ the integral on $A\subset[0,\infty)$ of a function $\theta$ with respect to the measure $g$, also if $g$ is not absolutely continuous with respect to the Lebesgue measure. Notice that any multiple of the Dirac delta $\delta_0$ is a solution to \eqref{selfsim4} in the weak sense.

The precise result reads as follows.

\begin{theorem} \label{thm:nonexist}
Assume that the kernel $K$ satisfies assumptions \eqref{kernel1}, \eqref{kernel2}, \eqref{kernel4} and \eqref{kernel5}.
Then there exists $\rho_{**}\geq\rho_*$ such that for every $\rho>\rho_{**}$ and for $b=\frac{1+\rho}{\rho}$ the unique solution to \eqref{selfsim4} in $\meas$ with unit mass is the Dirac delta $\delta_0$.
\end{theorem}

The proof of the theorem, which is given in Section~\ref{sect:nonex}, is achieved through a contradiction argument which mainly relies on a duality formula. The idea starts with the observation that, if $g$ is a measure solution to \eqref{selfsim4}, then it is a weak stationary solution to the corresponding evolution equation
\begin{equation*}
	\partial_t g + b\partial_x(xg) - \partial_x \Bigl( \int_{0}^{x}\int_{x-y}^\infty \frac{K(y,z)}{z}g(y)g(z)\de z \de y \Biggr)=0\,.
\end{equation*}
By formally testing this equation with a function $\vphi$ chosen as a solution to the dual problem
\begin{equation} \label{adjoint0}
	\partial_t \vphi(x,t) + bx\partial_x\vphi(x,t) - \inte \frac{K(x,z)}{z}g(z)[\vphi(x+z,t)-\vphi(x,t)]\de z = 0
\end{equation}
with $\vphi(x,0)=\chi_{(R_0,\infty)}(x)$, for $R_0>0$, one then obtains that
\begin{equation} \label{duality0}
	\inte \vphi(x,T)g(x)\de x = \inte\vphi(x,0)g(x)\de x = \int_{(R_0,\infty)} g(x)\de x\,.
\end{equation}
To understand the behaviour of solutions to \eqref{adjoint0}, we use \eqref{kernel4} and a Taylor expansion of $\vphi$ in the integral term, and we see that for $x$ large the approximate form of \eqref{adjoint0} is
\begin{equation} \label{duality00}
\partial_t\vphi - \beta_0M_\alpha x^{1-\alpha}\partial_x\vphi + (b-1)x\partial_x\vphi = C \partial^2_{xx}\vphi\,,
\end{equation}
where $M_\alpha=\inte x^\alpha g(x)\de x$ denotes the $\alpha$-moment of $g$.
Given $R_0>0$ one can then choose $b$ sufficiently close to 1 so that the effect of the transport terms on the left-hand side of \eqref{duality00} is to move most of the mass towards the origin: in this case $\vphi(x,T)$ would become uniformly positive for large times, giving a contradiction with \eqref{duality0} if $R_0$ is large enough. The picture is actually made more difficult by the presence of the diffusive term on the right-hand side, which one has to control. In fact, in order to avoid the development of a well-posedness and regularity theory for the adjoint equation \eqref{adjoint0}, we will actually construct an explicit subsolution to \eqref{adjoint0} (Lemma~\ref{lem:subsolution}) which will allow us to exploit the basic idea contained in the previous formal argument.

\begin{remark} \label{rm:beta0}
The assumption that $\beta_0>0$ in \eqref{kernel4} is crucial in our argument for the proof of Theorem~\ref{thm:nonexist}: this can be seen from the approximate form \eqref{duality00} of the adjoint equation, in which the sign of $\beta_0$ determines the sign of the main transport term. In fact, in the case $\beta_0<0$ we expect that solutions to \eqref{selfsim4} with finite mass might exist also for $b\to1$; in Section~\ref{sect:euristics} we compute the formal asymptotics of a solution in this case, see \eqref{eq:beta0}.
\end{remark}

\begin{remark}\label{rm:general}
Unfortunately, our Theorems only say something about sufficiently small values of $\rho$ (Theorem~\ref{thm:exist}), or sufficiently large values of $\rho$ (Theorem~\ref{thm:nonexist}) respectively. One might expect that there is a single critical $\rho_{crit} \in (0,\infty]$, such that for each $\rho<\rho_{crit}$ a self-similar solution with power law decay exists, for $\rho=\rho_{crit}$ a solution with exponential decay, and for larger $\rho$ (in case that $\rho_{crit}<\infty$) no nonnegative solutions. We expect, however, that a proof of a corresponding statement, if true at all, is in general  difficult and is not feasible with the methods developed in this paper.
\end{remark}


\section{Heuristics of asymptotic behaviour} \label{sect:euristics}

We present a heuristic argument for \eqref{b}, which gives a one-to-one correspondence between the coefficient $b$ in the self-similar equation \eqref{selfsim4} and the exponent $\rho$ describing the asymptotic decay behaviour of a self-similar solution. We assume here that $K$ is a general kernel, homogeneous of degree one, which satisfies $K(\xi,1)\to 1 $ as $\xi\to0$. For this argument, it is convenient to reformulate the equation \eqref{selfsim4} by means of the change of variables
\begin{equation*} \label{heuristic1}
G(X) = xg(x), \qquad x=e^X\,.
\end{equation*}
Then \eqref{selfsim4} becomes
\begin{equation} \label{heuristic2}
bG(X) = \int_{-\infty}^X\int_{X+\ln(1-e^{Y-X})}^\infty K(e^{Y-Z},1)G(Y)G(Z)\de Z \de Y\,,
\end{equation}
with the total mass
\begin{equation*}
\int_{-\infty}^\infty G(X)\de X =M\,.
\end{equation*}
Notice that, in these variables, self-similar solutions to \eqref{smol} correspond to traveling wave solutions.
We now assume that a self-similar profile exists and we make the ansatz $G(X)\sim e^{-\rho X}$ as $X\to\infty$, which corresponds to $g(x)\sim x^{-(1+\rho)}$ as $x\to\infty$. We formally compute the asymptotics of the integral on the right-hand side of \eqref{heuristic2} for $X\to\infty$: we can split the region of integration into two parts, namely
\begin{align*}
\int_{-\infty}^X\int_{X+\ln(1-e^{Y-X})}^\infty & K(e^{Y-Z},1)G(Y)G(Z)\de Z \de Y\,
= \int_{-\infty}^X \int_{X}^\infty K(e^{Y-Z},1)G(Y)G(Z)\de Z \de Y \\
& \qquad + \int_{-\infty}^X\int_{X+\ln(1-e^{Y-X})}^X K(e^{Y-Z},1)G(Y)G(Z)\de Z \de Y
=: (I) + (II)\,.
\end{align*}
By using the assumption $K(\xi,1)\sim 1$ for small $\xi$, we have for the first term
\begin{equation*}
(I) \sim \int_{-\infty}^X G(Y)\de Y \int_{X}^\infty G(Z)\de Z \sim \frac{M}{\rho}e^{-\rho X}
\end{equation*}
as $X\to\infty$. For the second term, we exchange the order of integration and we approximate $G(Y)\sim G(X)$:
\begin{align*}
(II) \sim G(X) \int_{-\infty}^X \de Z \, G(Z) \int_{X+\ln(1-e^{Z-X})}^X K(e^{Y-Z},1)\de Y \sim G(X) \int_{-\infty}^X G(Z)\de Z \sim M e^{-\rho X}
\end{align*}
as $X\to\infty$. Here we used the fact that
\begin{align*}
 \int_{X+\ln(1-e^{Z-X})}^X K(e^{Y-Z},1)\de Y = \int_{x-z}^x K\Bigl(\frac{y}{z},1\Bigr)\frac{1}{y}\de y = \frac{1}{z}\int_{x-z}^x K\Bigl(1,\frac{z}{y}\Bigr) \de y\sim 1\,.
\end{align*}
By \eqref{heuristic2} we then deduce that, if a solution $G$ with $G(X)\sim e^{-\rho X}$ as $X\to\infty$ exists, then $b$ and $\rho$ are related by \eqref{b}.

We can similarly compute the asymptotics at $X\to-\infty$: assume that $G(X)\sim e^{aX}$ as $X\to-\infty$ for some $a>0$ (which corresponds to $g(x)\sim x^{a-1}$ as $x\to0$ in the old variables). Then, by splitting the right-hand side of \eqref{heuristic2} as before, we obtain by similar arguments
\begin{equation*}
(I) \sim \int_{-\infty}^X G(Y)\de Y \int_{X}^\infty G(Z)\de Z \sim \frac{M}{a}e^{aX}
\end{equation*}
as $X\to-\infty$, while $(II)$ is in this case a higher order term. Hence we conclude that $a=\frac{M}{b}=\frac{\rho}{1+\rho}$, that is, $g(x)\sim x^{-\frac{1}{1+\rho}}$ as $x\to0$.

\medskip
Assume now that the kernel $K$ satisfies \eqref{kernel4} with a coefficient $\beta_0$ of opposite sign:
\begin{equation} \label{heuristic3}
K(1,\xi) \sim 1 - \beta_0\xi^\alpha \qquad\text{as }\xi\to 0
\end{equation}
for $\beta_0>0$ and $\alpha\in(0,1)$.
By computations similar to those above, we can describe the expected decay behaviour of a solution for $b=1$, supporting the belief that for kernels satisfying \eqref{heuristic3} solutions might exist also for $b$ close to 1 (see Remark~\ref{rm:beta0}). Hence fix $b=1$ and assume to have a solution $G$ to \eqref{heuristic2} with unit mass; we can write as before
\begin{equation} \label{heuristic4}
G(X) = (I) + (II)\,.
\end{equation}
By using \eqref{heuristic3} we have for $X\to\infty$
\begin{align} \label{heuristic5}
(I) &\sim \int_{-\infty}^X G(Y)\de Y \int_X^\infty G(Z)\de Z - \beta_0\int_{-\infty}^X e^{\alpha Y}G(Y)\de Y \int_X^\infty e^{-\alpha Z}G(Z)\de Z \nonumber\\
& \sim \int_{X}^\infty G(Z)\de Z - \beta_0 M_\alpha\int_X^\infty e^{-\alpha Z}G(Z)\de Z\,,
\end{align}
where $M_\alpha = \int_{-\infty}^\infty e^{\alpha Y}G(Y)\de Y$. Observe that, by changing variables and using the homogeneity of the kernel, together with \eqref{heuristic3}, we have
\begin{align*}
\int_{X+\ln(1-e^{Z-X})}^X K(e^{Y-Z},1)\de Y
& = \frac{1}{z}\int_{x-z}^x K\Bigl( 1,\frac{z}{y} \Bigr) \de y
\sim 1 - \frac{\beta_0}{z^{1-\alpha}}\int_{x-z}^x y^{-\alpha}\de y \\
& \sim 1 - \beta_0 z^\alpha x^{-\alpha} = 1 -\beta_0 e^{\alpha(Z-X)}\,.
\end{align*}
By using this expression, we can approximate the integral in the second region by exchanging the order of integration and approximating $G(Y)\sim G(X)$:
\begin{align} \label{heuristic6}
(II) &\sim G(X)\int_{-\infty}^X \de Z\, G(Z) \int_{X+\ln(1-e^{Z-X})}^X K(e^{Y-Z},1)\de Y \nonumber \\
& \sim G(X)\int_{-\infty}^X G(Z)\de Z - \beta_0 e^{-\alpha X}G(X) \int_{-\infty}^X e^{\alpha Z}G(Z)\de Z \nonumber\\
& \sim G(X) - \beta_0 M_\alpha e^{-\alpha X}G(X)\,.
\end{align}
Collecting \eqref{heuristic4}--\eqref{heuristic6} we have
\begin{align*}
G(X) \sim \int_X^\infty G(Z)\de Z - \beta_0 M_\alpha \int_X^\infty e^{-\alpha Z}G(Z)\de Z + G(X) - \beta_0 M_\alpha e^{-\alpha X}G(X)\,.
\end{align*}
Differentiating,
\begin{align*}
G(X) \sim \beta_0 M_\alpha e^{-\alpha X}G(X) - \beta_0 M_\alpha \frac{\de}{\de X}\Bigl(e^{-\alpha X}G(X)\Bigr)\,.
\end{align*}
Notice that the first term on the right-hand side is a higher order term as $X\to\infty$, and we can neglect it; we then obtain the decay behaviour
\begin{equation*}
e^{-\alpha X}G(X) \sim e^{-\frac{1}{\beta_0\alpha M_\alpha}e^{\alpha X}}\,.
\end{equation*}
In the original variables, this corresponds to
\begin{equation} \label{eq:beta0}
g(x)\sim x^{\alpha-1}e^{-\frac{1}{\beta_0\alpha M_\alpha}x^\alpha} \qquad\text{as }x\to\infty.
\end{equation}


\section{Existence of self-similar solutions via fixed point} \label{sect:ex1}

In this and in the following section we give the proof of Theorem~\ref{thm:exist}. We henceforth assume that the kernel $K$ satisfies conditions \eqref{kernel1}--\eqref{kernel3}. We will also always assume without loss of generality that $\rho<1$.
It is convenient to reformulate the problem in a new set of variables: for $g$ solving \eqref{selfsim4}, we define
$$
\rho \lambda(x) := e^{\frac{x}{\rho}} g(e^{\frac{x}{\rho}}),
\qquad x\in(-\infty,\infty).
$$
Then equation \eqref{selfsim4} becomes
\begin{equation} \label{fixedpoint}
(1+\rho) \lambda(x) = \int_{-\infty}^x \int_{x+\rho\ln(1-e^\frac{y-x}{\rho})}^\infty K(e^\frac{y-z}{\rho},1)\lambda(y)\lambda(z)\de z \de y \,.
\end{equation}
Notice that the change of variables is mass-preserving. The goal is hence to show, for $\rho$ sufficiently small, the existence of a nonnegative solution $\lambda$ to \eqref{fixedpoint} with unit mass and asymptotic decay
\begin{equation} \label{asymp2}
\lambda(x)\sim e^\frac{x}{1+\rho} \quad\text{as }x\to-\infty,
\qquad
\lambda(x)\sim e^{-x} \quad\text{as }x\to\infty.
\end{equation}

The strategy to construct a solution with the required properties mainly relies on a fixed point argument. We first observe that equation \eqref{fixedpoint} can be reformulated as
\begin{align} \label{fixedpoint2}
(1+\rho)\lambda(x)
& = \int_{-\infty}^x \lambda(y)\de y \int_x^\infty \lambda(z)\de z
+ \int_{-\infty}^x \int_x^\infty \Bigl[K(e^\frac{y-z}{\rho},1)-1\Bigr] \lambda(y)\lambda(z)\de z \de y \nonumber\\
& \qquad + \int_{-\infty}^x\int_{x+\rho\ln(1-e^\frac{y-x}{\rho})}^{x} K(e^\frac{y-z}{\rho},1)\lambda(y)\lambda(z)\de z \de y\,.
\end{align}
The idea is to regard the last two integrals in \eqref{fixedpoint2} as remainder terms, and hence to look for a solution to \eqref{fixedpoint2} in the form $\lambda=\lambar+\psi$, that is as a perturbation of the function
\begin{equation} \label{lambdabar}
\lambar(x):= \frac{1}{1+\rho} \frac{e^\frac{x}{1+\rho}}{(1+e^\frac{x}{1+\rho})^2}
\end{equation}
which is the explicit solution to
\begin{equation*}
(1+\rho)\lambar(x) = \int_{-\infty}^x\lambar(y)\de y \int_x^\infty \lambar(z)\de z\,,
\qquad\text{with }\int_{-\infty}^{\infty}\lambar(x)\de x =1\,.
\end{equation*}
By subtracting the equations for $\lambda$ and $\lambar$, we see that the unknown function $\psi$ should solve
\begin{equation} \label{eqtnpsi}
(1+\rho)\psi(x) = \int_{-\infty}^x\int_x^\infty \Bigl( \lambar(y)\psi(z) + \psi(y)\lambar(z) \Bigr) \de z \de y + \mathcal{R}_\rho[\psi](x)\,,
\end{equation}
where the remainder term $\mathcal{R}_\rho[\psi]$ is given by
\begin{align} \label{remainder}
\mathcal{R}_\rho[\psi](x) &:= \int_{-\infty}^x \psi(y)\de y \int_x^\infty \psi(z)\de z \nonumber\\
& \quad + \int_{-\infty}^x\int_x^\infty \Bigl[ K(e^\frac{y-z}{\rho},1)-1 \Bigr] \bigl(\lambar+\psi\bigr)(y) \bigl(\lambar+\psi\bigr)(z)\de z \de y \nonumber\\
& \quad + \int_{-\infty}^x \int_{x+\rho\ln(1-e^\frac{y-x}{\rho})}^x K(e^\frac{y-z}{\rho},1) \bigl(\lambar+\psi\bigr)(y) \bigl(\lambar+\psi\bigr)(z)\de z \de y \nonumber\\
& =: R_{1,\rho}[\psi](x) + R_{2,\rho}[\psi](x) + R_{3,\rho}[\psi](x)\,.
\end{align}
Since $\int_{-\infty}^\infty\psi(x)\de x =0$, the integral function $\Psi(x):=\int_{-\infty}^x\psi(y)\de y$ then satisfies the equation
\begin{equation*}
(1+\rho)\Psi'(x) - \frac{(1-e^\frac{x}{1+\rho})}{(1+e^\frac{x}{1+\rho})}\,\Psi(x) = \mathcal{R}_\rho[\psi](x)
\end{equation*}
which has the solution
\begin{equation} \label{eqtnPsi}
\Psi(x) = \frac{\lambar(x)}{1+\rho}\int_0^x \frac{\mathcal{R}_\rho[\psi](y)}{\lambar(y)}\de y\,.
\end{equation}
By differentiating $\Psi$ we hence obtain that a solution $\psi$ to \eqref{eqtnpsi} is implicitly given by the fixed point problem $\psi = \mathcal{H}_\rho[\psi]$, where
\begin{equation} \label{mapfixedpoint}
\mathcal{H}_\rho[\psi](x) := \frac{\lambar(x)}{(1+\rho)^2} \frac{(1-e^\frac{x}{1+\rho})}{(1+e^\frac{x}{1+\rho})} \int_0^x \frac{\mathcal{R}_\rho[\psi](y)}{\lambar(y)}\de y + \frac{1}{(1+\rho)}\,\mathcal{R}_\rho[\psi](x)\,.
\end{equation}

In the following proposition we show that the map $\mathcal{H}_\rho$ has indeed a fixed point in the space
\begin{equation} \label{spacefixedpoint}
\mathcal{X}_{\rho,\e} :=
\Bigl\{ \psi\in C(\R) \;:\; \int_{-\infty}^\infty\psi(x)\de x =0,\;
|\psi(x)|\leq \e e^{\gamma(x)}\, \Bigr\}\,,
\qquad
\|\psi\| := \sup_{x\in\R}\frac{|\psi(x)|}{\e e^{\gamma(x)}}\,,
\end{equation}
for $\e>0$, $\rho>0$ sufficiently small, where
\begin{equation} \label{gamma}
\gamma(x) :=
\begin{cases}
\frac{x}{1+\rho} & \text{if } x<0,\\
-\beta x & \text{if }x\geq 0,
\end{cases}
\end{equation}
and $\beta\in(\frac12,\frac{1}{1+\rho})$ is a fixed parameter.

\begin{proposition} \label{prop:fixedpoint}
There exist $\e>0$, $\rho_1>0$ such that for every $\rho\in(0,\rho_1)$ there is $\psi\in\mathcal{X}_{\rho,\e}$ such that $\mathcal{H}_\rho[\psi]=\psi$.
\end{proposition}

\begin{proof}
The goal is to prove that the operator $\mathcal{H}_\rho$ maps the space $\mathcal{X}_{\rho,\e}$ into itself and is strongly contractive if $\rho$ and $\e$ are small enough:
\begin{enumerate}
	\item\label{item1fp} $\mathcal{H}_\rho[\psi]\in\mathcal{X}_{\rho,\e}$ for every $\psi\in\mathcal{X}_{\rho,\e}$,
	\medskip
	\item\label{item2fp} $\| \mathcal{H}_\rho[\psi_1] - \mathcal{H}_\rho[\psi_2] \| \leq \sigma \|\psi_1-\psi_2\|$ for every $\psi_1,\psi_2\in\mathcal{X_{\rho,\e}}$, for some $\sigma\in(0,1)$.
\end{enumerate}
Banach's fixed point theorem will then imply the conclusion. Along the proof, we will denote by $C$ a generic constant depending possibly only on the fixed constants $K_0$, $\alpha$ appearing in \eqref{kernel3} and on $\beta$, but uniform with respect to $\rho$ and $\e$, which may change from line to line.

\medskip\noindent
\textit{Step 1: estimates on $\mathcal{R}_\rho[\psi]$.} We start by proving some preliminary estimates on the remainder term \eqref{remainder}. For the first integral in \eqref{remainder}, we have for all $x>0$
\begin{align} \label{fp1}
|R_{1,\rho}[\psi](x)|
&= \bigg| \int_{-\infty}^x \psi(y)\de y \int_x^\infty\psi(z)\de z \bigg|
= \bigg| \int_x^\infty \psi(z)\de z \bigg|^2 \nonumber\\
& \leq \e^2 \biggl( \int_x^\infty e^{\gamma(z)}\de z \biggr)^2
\leq C\e^2 e^{2\gamma(x)}\,,
\end{align}
while for $x<0$
\begin{align} \label{fp2}
|R_{1,\rho}[\psi](x)|
= \bigg| \int_{-\infty}^x \psi(z)\de z \bigg|^2
\leq \e^2 \biggl( \int_{-\infty}^x e^{\gamma(z)}\de z \biggr)^2
\leq C\e^2 e^{2\gamma(x)}\,.
\end{align}

\begin{figure}
	\centering
	\includegraphics{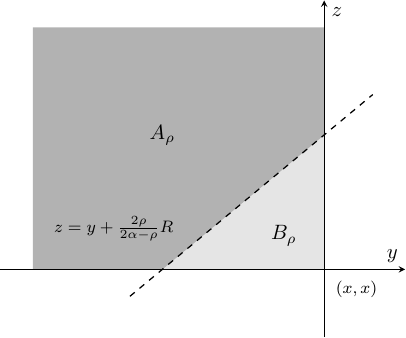}
	\hspace{4ex}
	\includegraphics{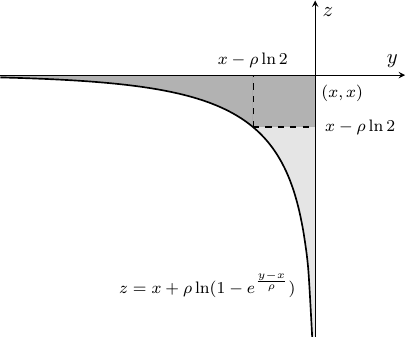}
	\caption{Left: the domain of integration in \eqref{fp3}, split into the two regions $A_\rho$, $B_\rho$. Right: the domain of integration in \eqref{fp7}.}
	\label{fig:domains}
\end{figure}

To estimate the second integral in \eqref{remainder}, we use the assumption \eqref{kernel3} together with the trivial bound $\lambar(x)\leq e^{\gamma(x)}$:
\begin{align} \label{fp3}
|R_{2,\rho}[\psi](x)|
&= \bigg| \int_{-\infty}^x \int_x^\infty \Bigl[ K(e^\frac{y-z}{\rho},1)-1 \Bigr] \bigl(\lambar+\psi\bigr)(y) \bigl(\lambar+\psi\bigr)(z)\de z \de y \bigg| \nonumber\\
&\leq C \int_{-\infty}^x\int_x^\infty e^{\frac{\alpha}{\rho}(y-z)} e^{\gamma(y)}e^{\gamma(z)} \de z \de y \nonumber\\
&= C\iint_{A_\rho} e^{\frac{\alpha}{\rho}(y-z)} e^{\gamma(y)}e^{\gamma(z)} \de y \de z
+ C\iint_{B_\rho} e^{\frac{\alpha}{\rho}(y-z)} e^{\gamma(y)}e^{\gamma(z)} \de y \de z\,,
\end{align}
where we split the domain of integration into the two regions $A_\rho$ and $B_\rho$ defined as
\begin{align*}
A_\rho := \bigl\{ (y,z)\in\R^2 \;:\; y\leq x,\, z\geq x,\, z \geq y+ \textstyle\frac{2\rho}{2\alpha-\rho}R  \bigr\}\,, \\
B_\rho := \bigl\{ (y,z)\in\R^2 \;:\; y\leq x,\, z\geq x,\, z<y+ \textstyle\frac{2\rho}{2\alpha-\rho}R  \bigr\}\,,
\end{align*}
see Figure~\ref{fig:domains}, left. Here $R>0$ is to be chosen later and $\alpha$ is the  coefficient appearing in \eqref{kernel3}; notice that $2\alpha-\rho>0$ by choosing $\rho$ sufficiently small. We also remark that, for $R$ fixed, $|B_\rho|\to0$ as $\rho\to0$. With this choice we have
\begin{equation*}
e^{\frac{\alpha}{\rho}(y-z)} = e^{\frac12(y-z)} e^{\frac{2\alpha-\rho}{2\rho}(y-z)} \leq e^{-R}e^{\frac12(y-z)}
\qquad\text{if }(y,z)\in A_\rho,
\end{equation*}
which yields
\begin{align} \label{fp4}
\iint_{A_\rho} e^{\frac{\alpha}{\rho}(y-z)} e^{\gamma(y)}e^{\gamma(z)} \de y \de z
&\leq e^{-R} \int_{-\infty}^x e^{\frac12y+\gamma(y)}\de y \int_x^\infty e^{-\frac12z+\gamma(z)}\de z \nonumber\\
&\leq Ce^{-R}e^{-\frac12|x|}e^{\gamma(x)}\,.
\end{align}
Moreover, we claim that if $\rho$ is sufficiently small (depending on $R$), we have
\begin{equation} \label{fp5bis}
e^{\gamma(y)}e^{\gamma(z)} \leq 2 e^{2\gamma(x)}
\qquad\text{if }(y,z)\in B_\rho.
\end{equation}
Indeed, consider first the case $x>0$: we have that $e^{\gamma(z)}\leq e^{\gamma(x)}$ since $z>x$; moreover $y\geq x - \frac{2\rho}{2\alpha-\rho}R$, and by choosing $\rho$ (depending on $R$) so that the quantity $\frac{2\rho}{2\alpha-\rho}R$ is small enough, we also have $e^{\gamma(y)}\leq 2e^{\gamma(x)}$.
This shows that \eqref{fp5bis} holds if $x>0$, and the case $x<0$ is analogous.
Hence by \eqref{fp5bis}
\begin{align} \label{fp5}
\iint_{B_\rho} e^{\frac{\alpha}{\rho}(y-z)} e^{\gamma(y)}e^{\gamma(z)} \de y \de z
\leq \iint_{B_\rho} e^{\gamma(y)}e^{\gamma(z)}\de y \de z
\leq C|B_\rho|e^{2\gamma(x)}\,.
\end{align}
At this point, we can first choose $R$ in order to make $e^{-R}$ small, and consequently we can choose $\rho$ small (depending on $R$) such that \eqref{fp5bis} holds and $|B_\rho|$ is small. In this way, combining \eqref{fp4} and \eqref{fp5}, we obtain from \eqref{fp3} that
\begin{equation} \label{fp6}
|R_{2,\rho}[\psi](x)| \leq C(\rho) e^{-\frac12|x|}e^{\gamma(x)}\,,
\end{equation}
where $C(\rho)$ is a constant such that $C(\rho)\to0$ as $\rho\to0$.

Finally, we consider the third term in \eqref{remainder}:
\begin{align} \label{fp7}
|R_{3,\rho}[\psi](x)| \leq
C \int_{-\infty}^x \de z \int_{x+\rho\ln(1-e^\frac{z-x}{\rho})}^x \de y \, K(e^\frac{y-z}{\rho},1) e^{\gamma(y)}e^{\gamma(z)}\,.
\end{align}
We split also this integral in two regions, the one where $z<x-\rho\ln 2$ and its complement (see Figure~\ref{fig:domains}, right).
We also have to distinguish the cases $x>0$ and $x<0$. For $x>0$,
\begin{align} \label{fp8}
\int_{-\infty}^{x-\rho\ln2} \de z & \int_{x+\rho\ln(1-e^\frac{z-x}{\rho})}^x \de y \, K(e^\frac{y-z}{\rho},1) e^{\gamma(y)}e^{\gamma(z)} \nonumber\\
& \xupref{kernel2}{=} \int_{-\infty}^{x-\rho\ln2} \de z \int_{x+\rho\ln(1-e^\frac{z-x}{\rho})}^x \de y \, K(1, e^\frac{z-y}{\rho})e^\frac{y-z}{\rho} e^{\gamma(y)}e^{\gamma(z)} \nonumber\\
& \xupref{kernel3}{\leq} C \int_{-\infty}^{x-\rho\ln2} \de z \, e^{-\frac{z}{\rho}}e^{\gamma(z)} \int_{x+\rho\ln(1-e^\frac{z-x}{\rho})}^x e^{\frac{y}{\rho}} e^{\gamma(y)}\de y \nonumber\\
& \leq C e^{\gamma(x)} \int_{-\infty}^{x-\rho\ln2} \de z \, e^{-\frac{z}{\rho}}e^{\gamma(z)} \int_{x+\rho\ln(1-e^\frac{z-x}{\rho})}^x e^{\frac{y}{\rho}}\de y \nonumber\\
& = C\rho e^{\gamma(x)}\int_{-\infty}^{x-\rho\ln 2} e^{\gamma(z)}\de z
\leq C\rho e^{\gamma(x)}
\qquad\qquad\qquad (x>0),
\end{align}
where in the second inequality we also used the fact that $e^{\gamma(y)}\leq C e^{\gamma(x)}$ in the region of integration, since $y\in(x-\rho\ln 2, x)$.
The same argument gives for $x<0$ the better decay
\begin{align} \label{fp9}
\int_{-\infty}^{x-\rho\ln2}\de z & \int_{x+\rho\ln(1-e^\frac{z-x}{\rho})}^x \de y \, K(e^\frac{y-z}{\rho},1) e^{\gamma(y)}e^{\gamma(z)} \nonumber\\
& \leq C\rho e^{\gamma(x)}\int_{-\infty}^{x-\rho\ln 2} e^{\gamma(z)}\de z
\leq C\rho e^{2\gamma(x)}
\qquad\qquad\qquad (x<0).
\end{align}
We now turn to the integral in the region $x-\rho\ln 2 \leq z \leq x$.
Notice that, by continuity of $K$,
\begin{equation*}
K(e^\frac{y-z}{\rho},1)\leq C\qquad\text{if }y\leq x \text{ and } z\geq x-\rho\ln2.
\end{equation*}
Hence, in the case of $x>0$ we have
\begin{align} \label{fp10}
\int_{x-\rho\ln2}^x \de z & \int_{x+\rho\ln(1-e^\frac{z-x}{\rho})}^x \de y \, K(e^\frac{y-z}{\rho},1) e^{\gamma(y)}e^{\gamma(z)} \nonumber\\
& \leq C \int_{x-\rho\ln2}^x \de z \, e^{\gamma(z)} \int_{x+\rho\ln(1-e^\frac{z-x}{\rho})}^x e^{\gamma(y)}\de y \nonumber\\
& \leq C \rho e^{\gamma(x)} \biggl(-\int_{x-\rho\ln2}^x \ln(1-e^\frac{z-x}{\rho})\de z\biggr) \nonumber\\
& = C \rho^2 e^{\gamma(x)} \biggl(-\int_{-\ln 2}^{0} \ln(1-e^w)\de w\biggr) \leq C \rho^2 e^{\gamma(x)}
\qquad\qquad\qquad (x>0),
\end{align}
where in the second inequality we used the fact that $e^{\gamma(z)}\leq C e^{\gamma(x)}$ for $z\in(x-\rho\ln2,x)$ and $e^{\gamma(y)}\leq 1$.
Arguing similarly in the case $x<0$, and observing that in this case $e^{\gamma(y)}e^{\gamma(z)}\leq e^{2\gamma(x)}$ since $y,z\leq x<0$, we have
\begin{align} \label{fp11}
\int_{x-\rho\ln2}^x \de z & \int_{x+\rho\ln(1-e^\frac{z-x}{\rho})}^x \de y \, K(e^\frac{y-z}{\rho},1) e^{\gamma(y)}e^{\gamma(z)} \nonumber\\
& \leq C e^{2\gamma(x)} \int_{x-\rho\ln2}^x \de z \int_{x+\rho\ln(1-e^\frac{z-x}{\rho})}^x \de y
\leq C \rho^2 e^{2\gamma(x)}
\qquad (x<0).
\end{align}
Collecting \eqref{fp8}--\eqref{fp11} and inserting them in \eqref{fp7} we get
\begin{equation} \label{fp12}
|R_{3,\rho}[\psi](x)| \leq
\begin{cases}
C\rho e^{2\gamma(x)} & \text{if }x<0,\\
C\rho e^{\gamma(x)} & \text{if }x>0.
\end{cases}
\end{equation}

Finally, bringing together \eqref{fp1}, \eqref{fp2}, \eqref{fp6} and \eqref{fp12} we get the desired estimate on the remainder term \eqref{remainder}:
\begin{equation} \label{fp14}
|\mathcal{R}_\rho[\psi](x)|\leq
\begin{cases}
C \bigl( \e^2 + \rho + C(\rho) \bigr) e^{\frac{x}{2}}e^{\frac{x}{1+\rho}} & \text{if }x<0,\\
C \bigl( \e^2 + \rho + C(\rho) \bigr) e^{-\beta x} & \text{if }x>0,
\end{cases}
\end{equation}
with $C(\rho)\to0$ as $\rho\to0$.

\medskip\noindent
\textit{Step 2: proof of \ref{item1fp}.}
We have to show that $|\mathcal{H}_\rho[\psi](x)|\leq \e e^{\gamma(x)}$ for all $x\in\R$.
Notice that by definition \eqref{lambdabar} of $\lambar$ we have
\begin{equation*}
\lambar(x) \leq e^{-\frac{|x|}{1+\rho}}\,, \qquad \frac{1}{\lambar(x)}\leq Ce^\frac{|x|}{1+\rho}\,.
\end{equation*}
It follows using \eqref{fp14} that for $x<0$
\begin{align} \label{fp15}
\bigg| \lambar(x) \int_0^x \frac{\mathcal{R}_\rho[\psi](y)}{\lambar(y)}\de y \bigg| 
& \leq C\bigl( \e^2 + \rho + C(\rho) \bigr)  e^{\frac{x}{1+\rho}} \int_x^0 e^{\frac{y}{2}}\de y \nonumber\\
& \leq C\bigl( \e^2 + \rho + C(\rho) \bigr)  e^{\frac{x}{1+\rho}}\,,
\end{align}
and similarly for $x>0$
\begin{align} \label{fp16}
\bigg| \lambar(x) \int_0^x \frac{\mathcal{R}_\rho[\psi](y)}{\lambar(y)}\de y \bigg| 
&\leq C\bigl( \e^2 + \rho + C(\rho) \bigr)  e^{-\frac{x}{1+\rho}} \int_0^x e^{(\frac{1}{1+\rho}-\beta) y}\de y \nonumber\\
&\leq C\bigl( \e^2 + \rho + C(\rho) \bigr)  e^{-\beta x} e^{-(\frac{1}{1+\rho}-\beta) x}\int_0^x e^{(\frac{1}{1+\rho}-\beta) y}\de y \nonumber\\
&\leq C\bigl( \e^2 + \rho + C(\rho) \bigr)  e^{-\beta x}\,.
\end{align}
By definition \eqref{mapfixedpoint} of $\mathcal{H}_\rho$ we hence obtain from \eqref{fp14}, \eqref{fp15}, \eqref{fp16}
\begin{equation} \label{fp17}
|\mathcal{H}_\rho[\psi](x)| \leq
\begin{cases}
 C\bigl( \e^2 + \rho + C(\rho) \bigr) e^{\frac{x}{1+\rho}} & \text{if }x<0,\\
 C\bigl( \e^2 + \rho + C(\rho) \bigr) e^{-\beta x} & \text{if }x>0.
\end{cases}
\end{equation}
This yields $|\mathcal{H}_\rho[\psi](x)|\leq\e e^{\gamma(x)}$ provided $\e$ and $\rho$ are sufficiently small, as claimed.

Moreover, since by construction $\int_{-\infty}^x \mathcal{H}_\rho[\psi](x)\de x = \Psi(x)$, where $\Psi$ is the function defined in \eqref{eqtnPsi}, with $\Psi(x)\to0$ as $|x|\to\infty$ by the previous estimates, we also have
$$
\int_{-\infty}^\infty \mathcal{H}_\rho[\psi](x)\de x =0\,.
$$
It follows that $\mathcal{H}_\rho[\psi]\in\mathcal{X}_{\rho,\e}$ for all $\psi\in\mathcal{X}_{\rho,\e}$.

\medskip\noindent
\textit{Step 3: proof of \ref{item2fp}.}
Let $\psi_1,\psi_2\in\mathcal{X}_{\rho,\e}$ and set
\begin{equation*}
M := \|\psi_1-\psi_2\| = \sup_{x\in\R}\frac{|\psi_1(x)-\psi_2(x)|}{\e e^{\gamma(x)}}\,.
\end{equation*}
We first deduce a bound on the difference $\big| \mathcal{R}_\rho[\psi_1](x) - \mathcal{R}_\rho[\psi_2](x) \big|$, by estimates similar to those used in the first step of the proof. Indeed, using the fact that $\int_{-\infty}^\infty\psi_i=0$, we have for every $x<0$
\begin{align} \label{fp18}
\big| R_{1,\rho}[\psi_1](x) - R_{1,\rho}[\psi_2](x) \big|
& = \bigg| \biggl(\int_{-\infty}^x\psi_1(y)\de y\biggr)^2 -  \biggl( \int_{-\infty}^x \psi_2(y)\de y\biggr)^2 \bigg| \nonumber\\
& \leq \int_{-\infty}^x |\psi_1-\psi_2|(y)\de y \int_{-\infty}^x |\psi_1+\psi_2|(y)\de y \nonumber\\
& \leq 2\e^2 M \biggl( \int_{-\infty}^x e^{\gamma(y)}\de y \biggr)^2 \leq C\e^2Me^{2\gamma(x)}\,,
\end{align}
and similarly for $x>0$
\begin{align} \label{fp19}
\big| R_{1,\rho}[\psi_1](x) - R_{1,\rho}[\psi_2](x) \big|
= \bigg| \biggl(\int_{x}^\infty\psi_1(y)\de y\biggr)^2 -  \biggl( \int_{x}^\infty \psi_2(y)\de y\biggr)^2 \bigg|
\leq C\e^2Me^{2\gamma(x)}\,.
\end{align}
Observe now that
\begin{align} \label{fp20}
\big| \bigl(\lambar+\psi_1\bigr)(y) &\bigl(\lambar+\psi_1\bigr)(z) - \bigl(\lambar+\psi_2\bigr)(y) \bigl(\lambar+\psi_2\bigr)(z) \big| \nonumber\\
& = \big| \bigl(\lambar+\psi_1\bigr)(y) \bigl(\psi_1-\psi_2\bigr)(z) + \bigl(\lambar+\psi_2\bigr)(z) \bigl(\psi_1-\psi_2\bigr)(y) \big| \nonumber\\
&\leq C\e M e^{\gamma(y)} e^{\gamma(z)}\,.
\end{align}
By using \eqref{fp20}, we can bound the second term $R_{2,\rho}$ by
\begin{align} \label{fp21}
\big| R_{2,\rho}[\psi_1](x) - R_{2,\rho}[\psi_2](x) \big|
& \leq C\e M \int_{-\infty}^x \int_x^\infty \big| K(e^\frac{y-z}{\rho},1)-1 \big| e^{\gamma(y)}e^{\gamma(z)}\de z \de y \nonumber\\
& \leq \e M C(\rho) e^{-\frac{|x|}{2}}e^{\gamma(x)}
\end{align}
for some constant $C(\rho)\to0$ as $\rho\to0$,
where the second inequality is obtained by the same estimates used to prove \eqref{fp6}.
Similarly, using again \eqref{fp20} we can bound the third term $R_{3,\rho}$ by
\begin{align*}
\big| R_{3,\rho}[\psi_1](x) - R_{3,\rho}[\psi_2](x) \big|
\leq C\e M \int_{-\infty}^x \int_{x+\rho\ln(1-e^\frac{y-x}{\rho})}^x K(e^\frac{y-z}{\rho},1) e^{\gamma(y)}e^{\gamma(z)}\de z \de y
\end{align*}
from which we get, by the same estimates yielding \eqref{fp12},
\begin{equation} \label{fp22}
\big| R_{3,\rho}[\psi_1](x) - R_{3,\rho}[\psi_2](x) \big| \leq
\begin{cases}
C\e M \rho e^{2\gamma(x)} & \text{if }x<0,\\
C\e M \rho e^{\gamma(x)} & \text{if }x>0.
\end{cases}
\end{equation}
Collecting \eqref{fp18}, \eqref{fp19}, \eqref{fp21} and \eqref{fp22} we finally obtain
\begin{equation} \label{fp23}
\big| \mathcal{R}_\rho[\psi_1](x) - \mathcal{R}_\rho[\psi_2](x) \big|
\leq
\begin{cases}
CM\e \bigl( \e + \rho + C(\rho) \bigr) e^{\frac{x}{2}}e^\frac{x}{1+\rho} & \text{if }x<0,\\
CM\e \bigl( \e + \rho + C(\rho) \bigr) e^{-\beta x} & \text{if }x>0,
\end{cases}
\end{equation}
hence, recalling the definition of $M:=\|\psi_1-\psi_2\|$ and of the norm in the space $\mathcal{X}_{\rho,\e}$,
\begin{equation} \label{fp24}
\| \mathcal{R}_\rho[\psi_1] - \mathcal{R}_\rho[\psi_2] \| \leq C \bigl(\e+\rho+C(\rho)\bigr) \|\psi_1-\psi_2\|\,.
\end{equation}

In turn, by using \eqref{fp23}--\eqref{fp24} and arguing similarly to \eqref{fp15}--\eqref{fp16}, we easily obtain
\begin{equation}
\| \mathcal{H}_\rho[\psi_1] - \mathcal{H}_\rho[\psi_2] \| \leq C \bigl( \e + \rho + C(\rho) \bigr) \|\psi_1-\psi_2\|\,,
\end{equation}
so that by choosing $\e$ and $\rho$ sufficiently small we conclude that the map $\mathcal{H}_\rho$ is a contraction in the space $\mathcal{X}_{\rho,\e}$.
\end{proof}

The fixed point obtained in Proposition~\ref{prop:fixedpoint} is a solution of \eqref{eqtnpsi} and, in turn, the map $\lambda:=\lambar+\psi$ is the sought solution to \eqref{fixedpoint}. Notice that such solution is positive for $x<0$ and already has the expected decay behaviour at $x\to-\infty$, but not at $x\to+\infty$: the second condition in \eqref{asymp2}, together with the positivity and the differentiability of the solution, will be established in the next section.

\begin{corollary} \label{cor:fixedpoint}
For every $\rho\in(0,\rho_1)$ there exists a solution $\lambda\in C(\R)$ to \eqref{fixedpoint}, which in addition satisfies
\begin{equation} \label{cor1}
\int_{-\infty}^\infty \lambda(x)\de x =1,
\end{equation}
\begin{equation} \label{cor2}
\frac{1}{16} e^\frac{x}{1+\rho} \leq \lambda(x) \leq 2 e^\frac{x}{1+\rho}\quad\text{for }x<0,
\end{equation}
\begin{equation} \label{cor3}
|\lambda(x)| \leq 2 e^{-\beta x} \quad\text{for all }x>0,
\end{equation}
where $\beta\in(\frac12,\frac{1}{1+\rho})$ is a fixed parameter.
\end{corollary}

\begin{proof}
The map $\lambda:=\lambar + \psi$, where $\psi$ is the fixed point constructed in Proposition~\ref{prop:fixedpoint}, satisfies all the required properties for $\e$ sufficiently small.
\end{proof}


\section{Decay behaviour of self-similar solutions} \label{sect:ex2}

In this section we complete the proof of Theorem~\ref{thm:exist} by showing that the solution $\lambda$ to \eqref{fixedpoint} constructed in Corollary~\ref{cor:fixedpoint} is of class $C^1$, strictly positive, and satisfies \eqref{asymp2}. We henceforth continue to assume conditions \eqref{kernel1}--\eqref{kernel3} on the kernel; also along this section we will denote by $C$ a uniform constant, dependent possibly only on the constants $K_0,\alpha$ appearing in \eqref{kernel3}, which may change from line to line.

The main idea to prove the decay $\lambda(x)\sim e^{-x}$ as $x\to\infty$ is to rewrite equation \eqref{fixedpoint} as
$$
\lambda(x)= \int_{-\infty}^x\lambda(y)\de y \int_x^\infty \lambda(z)\de z + \omega(x)\,,
$$
and to show that the remainder term $\omega(x)$ decays faster than $e^{-x}$ as $x\to\infty$. In turn, such decay can be obtained as a consequence of an estimate of the form $|\lambda'(x)|\leq C\lambda(x)$, which we will establish in Lemma~\ref{lem:continuation} below.

In what follows, it will be convenient to separate the region of integration in the right-hand side of \eqref{fixedpoint} into two parts:
\begin{align} \label{eqtndiff}
(1+\rho) \lambda(x)
= \int_{-\infty}^x \int_x^\infty K(e^\frac{y-z}{\rho},1)\lambda(y)\lambda(z)\de z \de y
+ \iint_{(x,x)+\Omega_\rho} K(e^\frac{y-z}{\rho},1)\lambda(y)\lambda(z) \de y\de z \,,
\end{align}
where the second integral is over the translation of the region
\begin{align*}
\Omega_\rho :
&= \bigl\{ (y,z)\in\R^2 \;:\; y\leq 0,\, z\leq0,\, e^\frac{y}{\rho} + e^\frac{z}{\rho} \geq 1  \bigr\}\,.
\end{align*}
As a first step towards the differentiability of $\lambda$, in the following lemma we establish uniform estimates on the difference quotients
\begin{equation} \label{diffquot}
D_h\lambda(x) := \frac{\lambda(x) - \lambda(x-h)}{h}\,.
\end{equation}

\begin{lemma} \label{lem:diffquot}
There exists $\rho_2>0$ such that, for every $\rho\in(0,\rho_2)$, the solution $\lambda$ to \eqref{fixedpoint} constructed in Corollary~\ref{cor:fixedpoint} satisfies
\begin{equation*}
|D_h\lambda(x)| \leq Ce^{\gamma(x)} =
\begin{cases}
Ce^\frac{x}{1+\rho} & \text{for }x<0\\
Ce^{-\beta x} & \text{for }x\geq 0
\end{cases}
\end{equation*}
for all $h\in(0,\rho)$, where $C>0$ is a constant independent of $\rho$ and $h$.
\end{lemma}

\begin{proof}
We set
\begin{equation*}
\Phi_h(x) := \sup_{y\in(x-\rho,x)}|D_h\lambda(y)|\,.
\end{equation*}
The conclusion will be achieved by proving a uniform decay estimate on $\Phi_h$ using an iteration argument.
First notice that, by \eqref{cor2}--\eqref{cor3}, $\Phi_h(x) \leq \frac{C}{h}e^{\gamma(x)}$, and hence
\begin{equation} \label{diffquot-1}
\lim_{|x|\to\infty}\Phi_h(x)=0
\end{equation}
(but the convergence is not uniform with respect to $h$). We now divide the proof into two steps.

\medskip\noindent
\textit{Step 1.} We claim that for all $x\in\R$
\begin{equation} \label{diffquot0}
|D_h\lambda(x)| \leq Ce^{\gamma(x)} + C\rho\Phi_h(x)\int_{-\infty}^x|\lambda(y)|\de y + C\rho e^{\gamma(x)}\int_{-\infty}^x |D_h\lambda(y)|\de y \,.
\end{equation}
To prove \eqref{diffquot0}, we apply the difference quotient operator $D_h$ to the equation \eqref{eqtndiff}:
\begin{align} \label{diffquot1}
(1+\rho)|D_h\lambda(x)|
& \leq \frac{1}{h}\bigg|\int_{x-h}^x \int_x^\infty K(e^\frac{y-z}{\rho},1)\lambda(y)\lambda(z)\de z \de y\bigg| \nonumber\\
& \qquad + \frac{1}{h}\bigg|\int_{-\infty}^{x-h} \int_{x-h}^x K(e^\frac{y-z}{\rho},1)\lambda(y)\lambda(z)\de z \de y \bigg| \nonumber\\
& \qquad + \frac{1}{h} \iint_{(x,x)+\Omega_\rho} K(e^\frac{y-z}{\rho},1) \big| \lambda(y)\lambda(z) - \lambda(y-h)\lambda(z-h)\big| \de y\de z \,.
\end{align}
The first term on the right-hand side of \eqref{diffquot1} can be estimated by
\begin{align} \label{diffquot2}
\frac{1}{h}\bigg|\int_{x-h}^x \int_x^\infty K(e^\frac{y-z}{\rho},1)\lambda(y)\lambda(z)\de z \de y\bigg|
& \xupref{kernel3}{\leq} \frac{C}{h}\int_{-\infty}^\infty|\lambda(z)|\de z \int_{x-h}^x |\lambda(y)|\de y
\leq Ce^{\gamma(x)}\,,
\end{align}
with the last inequality following easily from the estimates \eqref{cor2} and \eqref{cor3}.
Similarly, the second term in \eqref{diffquot1} is bounded by
\begin{align} \label{diffquot3}
\frac{1}{h}\bigg| \int_{-\infty}^{x-h}\int_{x-h}^x K(e^\frac{y-z}{\rho},1)\lambda(y)\lambda(z)\de z \de y\bigg|
\leq Ce^{\gamma(x)}\,.
\end{align}
It remains to estimate the third term in \eqref{diffquot1}, which can be written as
\begin{align} \label{diffquot4}
\frac{1}{h} \iint_{(x,x)+\Omega_\rho} K(e^\frac{y-z}{\rho},1) & \big| \lambda(y)\lambda(z) - \lambda(y-h)\lambda(z-h)\big| \de y\de z \nonumber\\
& \leq \iint_{(x,x)+\Omega_\rho} K(e^\frac{y-z}{\rho},1)|D_h\lambda(y)||\lambda(z)|\de y \de z \nonumber\\
& \qquad + \iint_{(x,x)+\Omega_\rho} K(e^\frac{y-z}{\rho},1)|\lambda(y-h)||D_h\lambda(z)|\de y \de z\,.
\end{align}
We split the region of integration into two parts: $\Omega_\rho= \Omega_\rho^+\cup\Omega_\rho^-$, where
$$
\Omega_\rho^- := \bigl\{ (y,z)\in\Omega_\rho \;:\; z<-\rho\ln 2 \bigr\}\,.
$$
Then
\begin{align} \label{diffquot5}
\iint_{(x,x)+\Omega_\rho^-} &K(e^\frac{y-z}{\rho},1)|D_h\lambda(y)||\lambda(z)|\de y \de z \nonumber\\
& = \int_{-\infty}^{x-\rho\ln 2} \de z |\lambda(z)| \int_{x+\rho\ln(1-e^\frac{z-x}{\rho})}^x K(1,e^\frac{z-y}{\rho}) e^\frac{y-z}{\rho} |D_h\lambda(y)|\de y \nonumber\\
& \leq C\Phi_h(x) \int_{-\infty}^{x-\rho\ln2} \de z |\lambda(z)|\int_{x+\rho\ln(1-e^\frac{z-x}{\rho})}^x e^\frac{y-z}{\rho}\de y \nonumber\\
& = \rho C\Phi_h(x) \int_{-\infty}^{x-\rho\ln2}|\lambda(z)|\de z\,,
\end{align}
where we used in particular \eqref{kernel2} and \eqref{kernel3}.
The same integral in the region $(x,x)+\Omega_\rho^+$ can be estimated by
\begin{align} \label{diffquot6}
\iint_{(x,x)+\Omega_\rho^+} &K(e^\frac{y-z}{\rho},1)|D_h\lambda(y)||\lambda(z)|\de y \de z \nonumber\\
& = \int_{x-\rho\ln2}^{x} \de z |\lambda(z)| \int_{x+\rho\ln(1-e^\frac{z-x}{\rho})}^x K(e^\frac{y-z}{\rho},1) |D_h\lambda(y)|\de y \nonumber\\
& \leq C\int_{x-\rho\ln2}^x |\lambda(z)|\de z \int_{-\infty}^x |D_h\lambda(y)|\de y \nonumber\\
& \leq C\rho e^{\gamma(x)}\int_{-\infty}^x|D_h\lambda(y)|\de y\,,
\end{align}
where we used the continuity of $K$ in the first inequality and we estimated $|\lambda(z)|$ by $Ce^{\gamma(z)}$, thanks to \eqref{cor2}--\eqref{cor3}, in the last one.
For the second term in \eqref{diffquot4}, we have
\begin{align} \label{diffquot7}
\iint_{(x,x)+\Omega_\rho^-} &K(e^\frac{y-z}{\rho},1)|\lambda(y-h)||D_h\lambda(z)|\de y \de z \nonumber\\
& = \int_{-\infty}^{x-\rho\ln 2} \de z |D_h\lambda(z)| \int_{x+\rho\ln(1-e^\frac{z-x}{\rho})}^x K(1,e^\frac{z-y}{\rho}) e^\frac{y-z}{\rho} |\lambda(y-h)|\de y \nonumber\\
& \leq Ce^{\gamma(x)}\int_{-\infty}^{x-\rho\ln2}\de z |D_h\lambda(z)|\int_{x+\rho\ln(1-e^\frac{z-x}{\rho})}^x e^\frac{y-z}{\rho}\de y \nonumber\\
& \leq C\rho e^{\gamma(x)}\int_{-\infty}^{x-\rho\ln2}|D_h\lambda(z)|\de z\,,
\end{align}
where we again used \eqref{kernel3} and \eqref{cor2}--\eqref{cor3} in the first inequality.
Similarly,
\begin{align} \label{diffquot8}
\iint_{(x,x)+\Omega_\rho^+} &K(e^\frac{y-z}{\rho},1)|\lambda(y-h)||D_h\lambda(z)|\de y \de z \nonumber\\
& = \int_{x-\rho\ln2}^{x} \de z |D_h\lambda(z)| \int_{x+\rho\ln(1-e^\frac{z-x}{\rho})}^x K(e^\frac{y-z}{\rho},1) |\lambda(y-h)|\de y \nonumber\\
& \leq C\rho\Phi_h(x) \int_{-\infty}^x|\lambda(y-h)|\de y\,.
\end{align}
Finally, collecting \eqref{diffquot2}--\eqref{diffquot8} and inserting them in \eqref{diffquot1}, we conclude that \eqref{diffquot0} holds.

\medskip\noindent
\textit{Step 2.} We now use \eqref{diffquot0} to get the conclusion by an iteration argument.
First observe that
\begin{equation*}
\int_{-\infty}^x |D_h\lambda(y)|\de y = \sum_{n=0}^\infty\int_{x-(n+1)\rho}^{x-n\rho} |D_h\lambda(y)|\de y
\leq \sum_{n=0}^\infty \Phi_h(x-n\rho)\rho\,,
\end{equation*}
so that inserting this inequality in \eqref{diffquot0} we get
\begin{equation*}
|D_h\lambda(x)| \leq Ce^{\gamma(x)} + C\rho\Phi_h(x)\int_{-\infty}^x|\lambda(y)|\de y + C\rho^2 e^{\gamma(x)}\sum_{n=0}^\infty \Phi_h(x-n\rho)\,.
\end{equation*}
In turn, using the fact that $\sup_{y\in(x-\rho,x)}\Phi_h(y)\leq\Phi_h(x)+\Phi_h(x-\rho)$, we deduce
\begin{equation*}
\Phi_h(x) \leq Ce^{\gamma(x)} + C\rho \bigl( \Phi_h(x) + \Phi_h(x-\rho) \bigr)\int_{-\infty}^x|\lambda(y)|\de y
+ C\rho^2e^{\gamma(x)}\sum_{n=0}^\infty \Phi_h(x-n\rho)\,,
\end{equation*}
and by choosing $\rho$ sufficiently small we can absorb the terms with $\Phi_h(x)$ in the left-hand side and obtain that for every $x\in\R$
\begin{equation} \label{diffquot9}
\Phi_h(x) \leq Ce^{\gamma(x)} + C\rho\Phi_h(x-\rho)\int_{-\infty}^x |\lambda(y)|\de y + C\rho^2e^{\gamma(x)}\sum_{n=1}^\infty \Phi_h(x-n\rho)\,.
\end{equation}

We now set $A_l^h:=\Phi_h(l\rho)$ for $l\in\Z$ and $h\in(0,\rho)$. By \eqref{diffquot9} computed at $x=l\rho$ we have for some uniform constant $\Cbar$ (independent of $\rho, h$ and $l$)
\begin{equation} \label{diffquot10}
A_l^h \leq \Cbar e^{\gamma(l\rho)} + \Cbar\rho A_{l-1}^h\int_{-\infty}^{l\rho}|\lambda(y)|\de y + \Cbar\rho^2e^{\gamma(l\rho)}\sum_{k=-\infty}^{l-1}A_k^h\,,
\end{equation}
and in particular, by \eqref{cor2}, for $l<0$
\begin{equation} \label{diffquot11}
A_l^h \leq \Cbar e^{\gamma(l\rho)} + \Cbar\rho A_{l-1}^he^{\gamma(l\rho)} + \Cbar\rho^2e^{\gamma(l\rho)}\sum_{k=-\infty}^{l-1}A_k^h\,.
\end{equation}
We claim that if $\rho$ is sufficiently small then
\begin{equation} \label{diffquot12}
A_l^h \leq 2\Cbar e^{\gamma(l\rho)}
\qquad\text{for all }h\in(0,\rho), l\in\Z.
\end{equation}

We prove \eqref{diffquot12} by induction on $l$.
Notice first that $A_k^h\to0$ as $k\to-\infty$ by \eqref{diffquot-1}, and that the series $\sum_{k=-\infty}^\infty\rho A_k^h$ converges: indeed
\begin{align*}
\sum_{k=-\infty}^\infty \rho A_k^h
&= \sum_{k=-\infty}^\infty \rho\Phi_h(k\rho)
\leq \frac{C}{h} \sum_{k=-\infty}^\infty \rho e^{\gamma(k\rho)}
= \frac{C}{h}\sum_{k=-\infty}^\infty \int_{k\rho}^{(k+1)\rho}e^{\gamma(k\rho)}
\leq \frac{C}{h} \int_{-\infty}^{\infty}e^{\gamma(x)}\de x <\infty\,.
\end{align*}
Hence there exists $l_0<0$, depending on $h$ and $\rho$, such that for every $l\leq l_0$ 
\begin{equation*}
A_l^h\leq \frac12, \qquad \sum_{k=-\infty}^{l_0}\rho A_k^h \leq \frac12\,.
\end{equation*}
Using these inequalities in \eqref{diffquot11}, we immediately get the claim \eqref{diffquot12} for all $l\leq l_0$.
We now check the induction step: assuming that \eqref{diffquot12} holds for all $l\leq\bar{l}$, for some $\bar{l}\in\Z$, let us prove that \eqref{diffquot12} holds also for $\bar{l}+1$: we have
\begin{align*}
\sum_{k=-\infty}^{\bar{l}} \rho A_k^h
\leq 2\Cbar\rho\sum_{k=-\infty}^{\bar{l}} e^{\gamma(k\rho)}
\leq 2\Cbar \sum_{k=-\infty}^{\infty}\rho e^{\gamma(k\rho)} \leq C\Cbar\,,
\end{align*}
for another universal constant $C$ independent of $\rho,h$ and $l$. Then by \eqref{diffquot10}
\begin{align*}
A_{\bar{l}+1}^h
& \leq \Cbar e^{\gamma((\bar{l}+1)\rho)} + C\Cbar\rho A_{\bar{l}}^h + \Cbar\rho e^{\gamma((\bar{l}+1)\rho)}\sum_{k=-\infty}^{\bar{l}}\rho A_k^h \\
& \leq \Bigl( \Cbar + \rho C \Cbar^2  \Bigr) e^{\gamma((\bar{l}+1)\rho)}
\leq 2\Cbar e^{\gamma((\bar{l}+1)\rho)}
\end{align*}
if we choose $\rho$ small enough. This completes the proof by induction of \eqref{diffquot12}.

The conclusion of the lemma follows immediately from \eqref{diffquot12} by observing that, for every $x\in\R$, choosing $l\in\Z$ such that $(l-1)\rho\leq x\leq l\rho$ we have
\begin{equation*}
|D_h\lambda(x)| \leq \Phi_h(l\rho) = A_l^h \leq 2\Cbar e^{\gamma(l\rho)} \leq 4\Cbar e^{\gamma(x)}
\end{equation*}
for all $\rho$ sufficiently small and $h\in(0,\rho)$.
\end{proof}

In the following lemma we show that, thanks to the equation satisfied by $\lambda$, the bound on the difference quotients proved in Lemma~\ref{lem:diffquot} turns into a bound on the derivative $\lambda'$ in terms of $\lambda$ itself. This immediately yields the positivity of $\lambda$, and, with some more work, the expected decay of the solution at $\infty$.

\begin{lemma} \label{lem:continuation}
There exists $\rho_3>0$ such that, for every $\rho\in(0,\rho_3)$, the solution $\lambda$ to \eqref{fixedpoint} constructed in Corollary~\ref{cor:fixedpoint} is of class $C^1(\R)$. Moreover,
\begin{equation*}
|\lambda'(x)| \leq C \lambda(x) \qquad\text{for all }x\in\R
\end{equation*}
for some constant $C>0$ (independent of $\rho$), and in particular $\lambda(x)>0$ for all $x\in\R$.
\end{lemma}

\begin{proof}
By a change of variables in the second integral in \eqref{eqtndiff}, we have that $\lambda$ solves
\begin{align} \label{cont1}
(1+\rho)\lambda(x) &=
\int_{-\infty}^x \int_{x}^\infty K(e^\frac{y-z}{\rho},1)\lambda(y)\lambda(z)\de z \de y \nonumber\\
& \qquad + \iint_{\Omega_\rho} K(e^\frac{y-z}{\rho},1)\lambda(y+x)\lambda(z+x)\de y \de z\,.
\end{align}
By Lemma~\ref{lem:diffquot} the function $\lambda$ is Lipschitz continuous, and hence differentiable almost everywhere, with $|\lambda'(x)|\leq Ce^{\gamma(x)}$ for almost every $x\in\R$.
This implies that the right-hand side in \eqref{cont1} is differentiable for every $x\in\R$, and in turn it follows that $\lambda\in C^1(\R)$ with
\begin{align} \label{cont2}
(1+\rho)\lambda'(x)
& = \lambda(x) \int_x^\infty K(e^\frac{x-z}{\rho},1)\lambda(z)\de z
- \lambda(x)\int_{-\infty}^x K(e^\frac{y-x}{\rho},1)\lambda(y)\de y \nonumber\\
& \qquad+ \iint_{(x,x)+\Omega_\rho} K(e^\frac{y-z}{\rho},1) \bigl( \lambda'(y)\lambda(z) + \lambda(y)\lambda'(z) \bigr) \de y \de z\,.
\end{align}

We now claim that we can find a constant $A>0$ such that for every $\rho$ small enough the following implication holds:
\begin{equation} \label{cont3bis}
\text{\itshape if $|\lambda'(y)|\leq A \lambda(y)$ for every $y\leq x$, for some $x\in\R$, then $|\lambda'(x)|\leq\frac{A}{2}\lambda(x)$.}
\end{equation}
In order to prove the claim, we estimate the right-hand side of \eqref{cont2} under the assumption that $|\lambda'(y)|\leq A \lambda(y)$ for every $y\leq x$. Notice that this assumption implies
\begin{equation} \label{cont4}
e^{-A(x-y)} \lambda(x) \leq \lambda(y) \leq e^{A(x-y)}\lambda(x)\qquad\text{for every }y\leq x\,.
\end{equation}
We first have, using assumption \eqref{kernel3},
\begin{align}\label{cont5}
\bigg| \lambda(x) \int_x^\infty K(e^\frac{x-z}{\rho},1)\lambda(z)\de z \bigg|
+ \bigg| \lambda(x) & \int_{-\infty}^x K(e^\frac{y-x}{\rho},1)\lambda(y)\de y \bigg| \nonumber\\
& \leq C \lambda(x)\int_{-\infty}^\infty|\lambda(z)|\de z\,.
\end{align}
To estimate the last integral in \eqref{cont2}, we split the domain $\Omega_\rho$ into the two regions where $z<x-\rho\ln2$ and $x-\rho\ln2<z<x$ respectively (see Figure~\ref{fig:domains}, right).
In the first case we have
\begin{align} \label{cont6}
\bigg| \int_{-\infty}^{x-\rho\ln 2}\de z & \int_{x+\rho\ln(1-e^\frac{z-x}{\rho})}^x K(e^\frac{y-z}{\rho},1) \bigl( \lambda'(y)\lambda(z) + \lambda(y)\lambda'(z) \bigr) \de y \bigg| \nonumber\\
& \upupref{kernel2}{kernel3}{\leq} C \int_{-\infty}^{x-\rho\ln2}\de z \int_{x+\rho\ln(1-e^\frac{z-x}{\rho})}^x e^\frac{y-z}{\rho} \big| \lambda'(y)\lambda(z) + \lambda(y)\lambda'(z) \big| \de y \nonumber\\
& \leq  2A C \int_{-\infty}^{x-\rho\ln2}\de z \int_{x+\rho\ln(1-e^\frac{z-x}{\rho})}^x e^\frac{y-z}{\rho} \lambda(y)\lambda(z) \de y \nonumber\\
& \leq 2A e^{A\rho} C \lambda(x) \int_{-\infty}^{x-\rho\ln2}\de z \, \lambda(z) \int_{x+\rho\ln(1-e^\frac{z-x}{\rho})}^x e^\frac{y-z}{\rho} \de y \nonumber\\
& = 2\rho A e^{A\rho} C \lambda(x) \int_{-\infty}^{x-\rho\ln2}\lambda(z)\de z\,,
\end{align}
where in the third inequality we used the fact that, since $y\in(x-\rho\ln2,x)\subset(x-\rho,x)$, then by \eqref{cont4} $\lambda(y)\leq e^{A\rho}\lambda(x)$. Similarly, for the integral in the region $x-\rho\ln2<z<x$ we have
\begin{align} \label{cont7}
\bigg| \int_{x-\rho\ln 2}^{x}\de z & \int_{x+\rho\ln(1-e^\frac{z-x}{\rho})}^x K(e^\frac{y-z}{\rho},1) \bigl( \lambda'(y)\lambda(z) + \lambda(y)\lambda'(z) \bigr) \de y \bigg| \nonumber\\
& \leq 2A C \int_{x-\rho\ln 2}^{x}\de z \int_{x+\rho\ln(1-e^\frac{z-x}{\rho})}^x \lambda(y)\lambda(z) \de y \nonumber\\
& \leq 2A e^{A\rho} C \lambda(x) \int_{x-\rho\ln 2}^{x}\de z \int_{x+\rho\ln(1-e^\frac{z-x}{\rho})}^x \lambda(y) \de y \nonumber\\
& \leq 2\rho A e^{A\rho} C \lambda(x) \int_{-\infty}^x \lambda(y)\de y\,,
\end{align}
where the second inequality follows as before from the fact that $\lambda(z)\leq e^{A\rho}\lambda(x)$ for $z\in(x-\rho\ln 2, x)$, by \eqref{cont4}.

Collecting \eqref{cont2}, \eqref{cont5}, \eqref{cont6} and \eqref{cont7} we conclude that
\begin{equation*}
|\lambda'(x)| \leq C\bigl( 1 + \rho A e^{A\rho} \bigr) \lambda(x)\,,
\end{equation*}
for some constant $C$ independent of $A$ and $\rho$.
It is then clear that we can choose the constant $A$ such that for $\rho$ sufficiently small we have $|\lambda'(x)|\leq \frac{A}{2}\lambda(x)$, which completes the proof of the implication \eqref{cont3bis}.

Since we already know that $|\lambda'(x)| \leq C \lambda(x)$ for every $x<0$, as a consequence of the bound $|\lambda'(x)|\leq Ce^{\gamma(x)}$ and of \eqref{cor2}, the conclusion of the lemma follows now easily from \eqref{cont3bis}.
\end{proof}

We are now in position to complete the proof of Theorem~\ref{thm:exist}.
By Lemma~\ref{lem:continuation}, the only property that we still have to check is the decay $\lambda(x)\sim e^{-x}$ as $x\to\infty$.

\begin{proof}[Proof of Theorem~\ref{thm:exist}]
By setting
\begin{align} \label{thm1}
\omega(x)
&:= \int_{-\infty}^x\int_x^\infty \Bigl[ K(e^\frac{y-z}{\rho},1)-1 \Bigr] \lambda(y)\lambda(z) \de y \de z \nonumber\\
& \qquad + \int_{-\infty}^x\de z \int_{x+\rho\ln(1-e^\frac{z-x}{\rho})}^x K(e^\frac{y-z}{\rho},1) \lambda(y)\lambda(z)\de y - \rho \lambda(x)\,,
\end{align}
the equation \eqref{fixedpoint} solved by $\lambda$ can be written in the form
\begin{equation} \label{thm1bis}
\lambda(x) = \int_{-\infty}^x\lambda(y)\de y \int_x^\infty \lambda(z)\de z + \omega(x)\,.
\end{equation}
The idea is that the decay behaviour of $\lambda$ for $x\to\infty$ is determined by the first term on the right-hand side of \eqref{thm1bis}, while the effect of the perturbation term $\omega(x)$ is negligible if $\rho$  is sufficiently small. In order to prove this rigorously, we need to show that $\omega(x)$ decays faster than $e^{-x}$ as $x\to\infty$. 

By using the fact that $\int_{-\infty}^\infty \lambda(z)\de z=1$, we can write $\omega(x)$ in the following way:
\begin{align}\label{thm2}
\omega(x)
& = \iint_{C_\rho} \Bigl[ K(e^\frac{y-z}{\rho},1)-1 \Bigr] \lambda(y)\lambda(z) \de y \de z \nonumber\\
& \qquad + \int_{x-\rho\ln2}^x \de z\,\lambda(z) \int_{x+\rho\ln(1-e^\frac{z-x}{\rho})}^x \lambda(y)\de y \nonumber\\
& \qquad + \iint_{D_\rho} \Bigl[ K(e^\frac{y-z}{\rho},1)-e^\frac{y-z}{\rho} \Bigr] \lambda(y)\lambda(z) \de y \de z \nonumber\\
& \qquad + \iint_{D_\rho}e^\frac{y-z}{\rho}\lambda(y)\lambda(z)\de y \de z
- \rho \lambda(x)\int_{-\infty}^{x-\rho\ln 2}\lambda(z)\de z \nonumber\\
& \qquad - \rho \lambda(x)\int_{x-\rho\ln 2}^\infty\lambda(z)\de z\,,
\end{align}
where (see Figure~\ref{fig:domains2})
\begin{align*}
C_\rho &:= \bigl\{ (y,z)\in\R^2 \;:\; x-\rho\ln2<z<x,\, x+\rho\ln(1-e^\frac{z-x}{\rho})<y<x \bigr\} \\
&\qquad\cup \bigl\{ (y,z)\in\R^2 \;:\; y\leq x,\,  z\geq x \bigr\}\,, \\
D_\rho &:= \bigl\{ (y,z)\in\R^2 \;:\; z < x-\rho\ln2,\, x+\rho\ln(1-e^\frac{z-x}{\rho})<y<x\bigr\}\,.
\end{align*}

\begin{figure}
	\centering
	\includegraphics{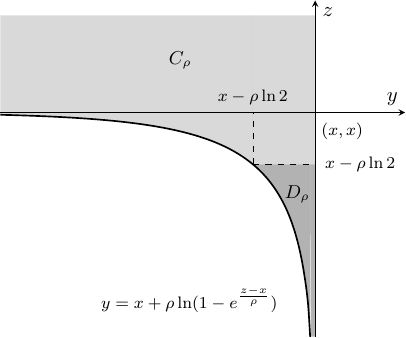}
	\caption{Representation of the domains of integration appearing in \eqref{thm2}.}
	\label{fig:domains2}
\end{figure}

Let us estimate each term in \eqref{thm2} for $x>0$. Recall that $\lambda(x)\leq 2e^{-\beta x}$ for $x>0$ by Corollary~\ref{cor:fixedpoint}, where $\beta\in(\frac12,\frac{1}{1+\rho})$ is a fixed parameter. We will assume that $\beta>\frac23$, which is a possible choice if $\rho<\frac12$.
We will often use the fact that, by Lemma~\ref{lem:continuation}, for all $x,y\in\R$
\begin{equation} \label{thm3}
\lambda(y) \leq e^{C|x-y|}\lambda(x)\,.
\end{equation}

To estimate the first integral in \eqref{thm2}, first notice that for $(y,z)\in C_\rho$ we have $y-z\leq\rho\ln2$, so that
$e^{\frac{\alpha}{\rho}(y-z)} = e^{\frac{y-z}{2}}e^{\frac{2\alpha-\rho}{2\rho}(y-z)} \leq Ce^{\frac{y-z}{2}}$,
where $\alpha$ is as in \eqref{kernel3} and we can assume $2\alpha-\rho>0$. It follows that
\begin{align} \label{thm4}
\bigg| \iint_{C_\rho} \Bigl[ K(e^\frac{y-z}{\rho},1)-1 \Bigr] \lambda(y)\lambda(z) \de y \de z \bigg|
&\xupref{kernel3}{\leq} C \iint_{C_\rho} e^{\frac{\alpha}{\rho}(y-z)}\lambda(y)\lambda(z)\de y \de z \nonumber\\
& \leq C\int_{x-\rho\ln 2}^\infty e^{-\frac{z}{2}}\lambda(z)\de z \int_{-\infty}^x e^{\frac{y}{2}} \lambda(y) \de y \nonumber\\
& \leq C\int_{x-\rho\ln 2}^\infty e^{-(\beta+\frac12)z}\de z \leq Ce^{-(\beta+\frac12)x}\,,
\end{align}
where in the third inequality we used the fact that the integral in the variable $y$ is bounded by a uniform constant $C$, due to the decay of $\lambda$.

For the second term in \eqref{thm2} we have using \eqref{thm3}
\begin{align} \label{thm5}
\int_{x-\rho\ln2}^x \de z\,\lambda(z) & \int_{x+\rho\ln(1-e^\frac{z-x}{\rho})}^x \lambda(y)\de y \nonumber\\
&\leq \int_{-\infty}^{x-\rho\ln2} \de y \,\lambda(y) \int_{x+\rho\ln(1-e^\frac{y-x}{\rho})}^x e^{C(x-z)}\lambda(x)\de z \nonumber\\
& \qquad + \int_{x-\rho\ln 2}^x\int_{x-\rho\ln2}^x e^{C(x-y)}e^{C(x-z)}\lambda(x)^2\de y \de z \nonumber\\
& \leq C \lambda(x)\int_{-\infty}^{x-\rho\ln 2}\lambda(y)(-\ln(1-e^\frac{y-x}{\rho}))\de y + C\lambda(x)^2 \nonumber\\
& \leq C \lambda(x)\int_{-\infty}^{x-\rho\ln 2}\lambda(y)e^\frac{y-x}{\rho}\de y + C e^{-2\beta x} \nonumber\\
& \leq C \lambda(x) \biggl( e^{-\frac{x}{2\rho}}\int_{-\infty}^{\frac{x}{2}}\lambda(y)\de y + \int_{\frac{x}{2}}^{x}e^{-\beta y}\de y\biggr) + C e^{-2\beta x} \nonumber\\
& \leq C e^{-(\beta+\frac{1}{2\rho})x} + C e^{-\frac32 \beta x} + C e^{-2\beta x}
\leq C e^{-\frac32\beta x}\,,
\end{align}
where in the third inequality we used $-\ln(1-t)\leq 2t$ for $t=e^\frac{y-x}{\rho}\in(0,\frac12)$.

By using \eqref{kernel2}, \eqref{kernel3} and \eqref{thm3} we bound the third term in \eqref{thm2} by
\begin{align} \label{thm6}
\bigg| \iint_{D_\rho} \Bigl[ K(e^\frac{y-z}{\rho},1)-&e^\frac{y-z}{\rho} \Bigr] \lambda(y)\lambda(z) \de y \de z \bigg|
\leq C \lambda(x) \iint_{D_\rho}e^\frac{y-z}{\rho}e^{\frac{\alpha}{\rho}(z-y)}e^{C(x-y)}\lambda(z) \de y \de z \nonumber\\
& = C\lambda(x)\int_{-\infty}^{x-\rho\ln 2} \de z\, \lambda(z)e^{\frac{\alpha-1}{\rho}z}e^{Cx} \int_{x+\rho\ln(1-e^\frac{z-x}{\rho})}^x e^{(\frac{1-\alpha}{\rho}-C)y}\de y \nonumber\\
& = \frac{\rho C\lambda(x)}{1-\alpha-C\rho} \int_{-\infty}^{x-\rho\ln 2} e^{\frac{1-\alpha}{\rho}(x-z)}\lambda(z) \bigl( 1-(1-e^\frac{z-x}{\rho})^{1-\alpha-C\rho} \bigr) \de z \nonumber\\
&\leq C\rho\lambda(x) \int_{-\infty}^{x-\rho\ln2}e^{\frac{\alpha}{\rho}(z-x)}\lambda(z)\de z \nonumber\\
&\leq C\lambda(x) \biggl( e^{-\frac{\alpha}{2\rho}x}\int_{-\infty}^{\frac{x}{2}} \lambda(z)\de z
+ \int_{\frac{x}{2}}^{x} e^{-\beta z}\de z\biggr) \nonumber\\
& \leq C\lambda(x)e^{-\frac{\beta}{2}x} \leq e^{-\frac{3}{2}\beta x}\,,
\end{align}
where in the second inequality we used $\frac{1}{1-\alpha-C\rho}(1-(1-t)^{1-\alpha-C\rho}) \leq Ct$ for $t=e^\frac{z-x}{\rho}\in(0,\frac12)$.

Observe now that for $(y,z)\in D_\rho$ we have $y\in(x-\rho,x)$ and in turn by Lemma~\ref{lem:continuation}
\begin{equation*}
|\lambda(y)-\lambda(x)| \leq \sup_{t\in(x-\rho,x)}|\lambda'(t)| (x-y)
\leq C \sup_{t\in(x-\rho,x)}\lambda(t) (x-y)
\leq C\lambda(x)(x-y)\,,
\end{equation*}
with the last inequality following from \eqref{thm3}.
Hence
\begin{align} \label{thm7}
\bigg| \iint_{D_\rho}e^\frac{y-z}{\rho}\lambda(y)\lambda(z)\de y \de z
& - \rho \lambda(x)\int_{-\infty}^{x-\rho\ln 2}\lambda(z)\de z \bigg| \nonumber\\
& = \bigg| \int_{-\infty}^{x-\rho\ln2} \de z \, \lambda(z) \int_{x+\rho\ln(1-e^\frac{z-x}{\rho})}^x e^\frac{y-z}{\rho} \bigl( \lambda(y) - \lambda(x) \bigr) \de y \bigg| \nonumber\\
& \leq C\lambda(x)\int_{-\infty}^{x-\rho\ln2}\de z \, \lambda(z) \int_{x+\rho\ln(1-e^\frac{z-x}{\rho})}^x e^\frac{y-z}{\rho}(x-y)\de y \nonumber\\
& \leq C\lambda(x) \int_{-\infty}^{x-\rho\ln2} \lambda(z) e^\frac{x-z}{\rho} \bigl( \rho\ln(1-e^\frac{z-x}{\rho}) \bigr)^2\de z \nonumber\\
& \leq C\lambda(x) \int_{-\infty}^{x-\rho\ln 2} e^\frac{z-x}{\rho}\lambda(z)\de z
\leq C e^{-\frac32\beta x}\,,
\end{align}
where the last inequality follows as in \eqref{thm5}.

Finally for the last term in \eqref{thm2} we have by \eqref{cor3}
\begin{align} \label{thm8}
\rho \lambda(x)\int_{x-\rho\ln 2}^\infty\lambda(z)\de z
\leq C\lambda(x)\int_{x-\rho\ln 2}^\infty e^{-\beta z}\de z
\leq C e^{-2\beta x}\,.
\end{align}

Eventually, collecting \eqref{thm4}--\eqref{thm8}, we conclude by \eqref{thm2} that for $x>0$
\begin{equation*}
|\omega(x)| \leq C e^{-\frac32\beta x}\,.
\end{equation*}
Since the parameter $\beta$ can be chosen such that $\frac32\beta>1$, by \eqref{thm1bis} we conclude that $\lambda(x)\sim e^{-x}$ as $x\to\infty$. This completes the proof of the theorem.
\end{proof}


\section{Nonexistence of self-similar solutions with fast decay} \label{sect:nonex}

In this section we give the proof of Theorem~\ref{thm:nonexist}, showing that, for $b$ sufficiently close to 1, \eqref{selfsim4} does not have solutions with finite mass in the class of positive measures, except for the Dirac delta at the origin. We henceforth assume that the kernel $K$ satisfies assumptions \eqref{kernel1}, \eqref{kernel2}, \eqref{kernel4} and \eqref{kernel5}.
We first give the precise definition of solution to the self-similar equation
\begin{equation} \label{selfsim4bis}
b x g(x) = \int_0^x\int_{x-y}^\infty \frac{K(y,z)}{z}g(y)g(z)\de z \de y
\end{equation}
in the sense of measures.

\begin{definition}[Weak solution] \label{def:weaksol}
We say that $g\in\meas$ is a weak solution to \eqref{selfsim4bis} if for every test function $\theta\in C([0,\infty))$ with compact support one has
\begin{equation} \label{eqtnmeas}
b \inte x\theta(x)g(x)\de x
= \inte \inte \frac{K(y,z)}{z}  g(y)g(z) \int_y^{y+z}\theta(x)\de x \de y \de z \,.
\end{equation}
For $b\geq1$, let
\begin{equation*}
\mathcal{S}_b := \Bigl\{ g\in\meas \,:\, g\text{ is a weak solution to \eqref{selfsim4bis}, } \int_{[0,\infty)}g(x)\de x =1,\, g\neq\delta_0 \Bigr\}
\end{equation*}
be the class of weak solutions to \eqref{selfsim4bis} with unit mass which are not concentrated at the origin.
\end{definition}

The proof of Theorem~\ref{thm:nonexist} amounts to show that $\mathcal{S}_b=\emptyset$ for all $b$ sufficiently close to 1.

\begin{remark} \label{rm:scaling}
Equation \eqref{eqtnmeas} has the following scale invariance property: if $g\in\meas$ is a weak solution to \eqref{selfsim4bis}, then for every $\lambda>0$ the rescaled measure $g_\lambda$ defined by
\begin{equation*}
\int_{[0,\infty)} \theta(x) g_\lambda(x) \de x = \int_{[0,\infty)} \theta\Bigl(\frac{y}{\lambda}\Bigr) g(y)\de y
\qquad\text{for all }\theta\in C([0,\infty))
\end{equation*}
is also a weak solution of \eqref{selfsim4bis}, with $\inte g_\lambda(x)\de x = \inte g(x)\de x$.
\end{remark}

\begin{remark} \label{rm:delta}
A weak solution to \eqref{selfsim4bis} which is different from $\delta_0$ cannot have a part concentrated at the origin: that is, for every $g\in\mathcal{S}_b$ one has
\begin{equation*}
\int_{\{0\}}g(x)\de x =0\,.
\end{equation*}
Indeed, assume by contradiction that $g\in\mathcal{S}_b$ has the form $g=m_0\delta_0 + \tilde{g}$ for some $m_0\in(0,1)$ and $\tilde{g}\in\meas$ with $\int_{\{0\}}\tilde{g}=0$. Then by \eqref{eqtnmeas} we obtain that for every $\theta\in C_{\mathrm c}([0,\infty))$
\begin{align*}
b \inte x\theta(x)\tilde{g}(x)\de x
&= m_0\inte\tilde{g}(z)\int_0^z\theta(x)\de x \de z
+ m_0\inte y\theta(y)\tilde{g}(y)\de y \\
& \qquad + \inte \inte \frac{K(y,z)}{z}  \tilde{g}(y)\tilde{g}(z) \int_y^{y+z}\theta(x)\de x \de y \de z \,,
\end{align*}
which yields, assuming $\theta\geq0$,
\begin{equation*}
(b-m_0) \inte x\theta(x)\tilde{g}(x)\de x \geq m_0\inte \tilde{g}(z)\int_0^z \theta(x)\de x \de z\,.
\end{equation*}
Choosing now $\theta(x):= (1-\frac{x}{\delta})\chi_{(0,\delta)}(x)$ for $\delta>0$, we easily obtain from the previous inequality
\begin{align*}
(b-m_0)\int_{(0,\delta)}\tilde{g}(x)\de x \geq \frac{m_0}{2}\int_{(\delta,\infty)}\tilde{g}(z)\de z\,,
\end{align*}
which is a contradiction if we choose $\delta$ small enough.
\end{remark}


\subsection{Decay of weak solutions}
As a preliminary step for the proof of Theorem~\ref{thm:nonexist}, we investigate here the decay properties of weak solutions to \eqref{selfsim4bis}.
We start by showing that, for $g$ in the class $\mathcal{S}_b$, the mass of $g$ in an interval of the form $(x,\infty)$ decays like a power law $x^{-\gamma}$ as $x\to\infty$, where the exponent $\gamma$ can be chosen arbitrarily large, provided $b$ is sufficiently close to 1.

\begin{lemma} \label{lem:decay}
Given any fixed $\gamma>0$, there exist constants $\eta_\gamma>0$, $b_\gamma>1$ and $C_\gamma>0$ such that for every $b\in[1,b_\gamma]$, $R_0>0$ and for every $g\in\mathcal{S}_b$, if
\begin{equation} \label{mass2}
\int_{[0,R_0]} g(x)\de x >1-\eta_\gamma
\end{equation}
then for every $x>0$
\begin{equation*}
\int_{[x,\infty)} g(y)\de y \leq C_\gamma \Bigl(\frac{R_0}{x}\Bigr)^\gamma\,.
\end{equation*}
\end{lemma}

\begin{proof}
In view of Remark~\ref{rm:scaling}, in proving the lemma we can assume without loss of generality that $R_0=1$.
Let $b>1$ and $\eta>0$ to be chosen later, and let $g\in\mathcal{S}_b$ satisfy \eqref{mass2} with $R_0=1$.
We set $\psi(x):=\int_{[x,\infty)}g(y)\de y$, so that $-\psi'=g$ as measures.
We fix $\e>0$ and let $\sigma=\sigma(\e)>0$ be such that
\begin{equation} \label{dec1}
|K(\xi,1)-1|<\e \qquad\text{if }\xi<\sigma.
\end{equation}
We also choose $x_0>\frac{1}{\e}$ such that $\frac{1}{x_0-1}<\sigma$.

By testing the equation \eqref{eqtnmeas} with a function $\theta\in C_{\mathrm c}(x_0,\infty)$, $\theta\geq0$, we have using \eqref{dec1}
\begin{align*}
- b\inte x\theta(x)\psi'(x) \de x
& \geq \int_{[0,1]} \de y \int_{[x_0-1,\infty)} \de z \, K\Bigl(\frac{y}{z},1\Bigr) g(y)g(z)\int_{y}^{y+z}\theta(x)\de x \\
& \quad + \int_{[0,1]}\de z \int_{[x_0-1,\infty)}\de y \, K\Bigl(1,\frac{z}{y}\Bigl)\frac{y}{z}\,g(y)g(z)\int_{y}^{y+z}\theta(x)\de x \\
& \geq (1-\e) \int_{[0,1]}\de y \int_{[x_0-1,\infty]}\de z \, g(y)g(z)\int_{y}^{y+z}\theta(x)\de x \\
& \quad + (1-\e) \int_{[0,1]}\de z \int_{[x_0-1,\infty)}\de y \, \frac{y}{z}\,g(y)g(z)\int_{y}^{y+z}\theta(x)\de x \\
& = (1-\e) \int_{[x_0,\infty)} \de x \int_{[0,1]}\de y \, \theta(x)g(y) \int_{[x-y,\infty)}g(z)\de z \\
& \quad + (1-\e) \int_{[x_0,\infty)} \de x \int_{[0,1]} \de z \,  \frac{1}{z}\,\theta(x)g(z)\int_{[x-z,x)}yg(y)\de y \\
& \geq (1-\e) \int_{[x_0,\infty)} \de x \int_{[0,1]} \theta(x)g(y)\psi(x-y) \de y\\
& \quad + (1-\e) \int_{[x_0,\infty)} \de x \int_{[0,1]} \frac{x-1}{z}\,\theta(x)g(z) \bigl( \psi(x-z)-\psi(x) \bigr) \de z\,.
\end{align*}
Observe that $x\geq x_0\geq \frac{1}{\e}$ implies $(x-1)\geq x(1-\e)$. By using also the monotonicity of $\psi$ and \eqref{mass2} we then obtain
\begin{align*}
b\inte x\theta(x)\psi'(x) \de x &+ (1-\e)^2 \inte \de x \int_{[0,1]} \frac{x}{z}\,\theta(x)g(z) \bigl( \psi(x-z)-\psi(x) \bigr) \de z\\
& \leq -(1-\e)(1-\eta) \inte \theta(x)\psi(x)\de x\,,
\end{align*}
that is, the function $\psi$ satisfies weakly in $(x_0,\infty)$ the inequality
\begin{equation} \label{dec2}
b \psi'(x) + (1-\e)^2 \int_{[0,1]}\frac{g(z)}{z} \bigl( \psi(x-z)-\psi(x) \bigr) \de z 
\leq -(1-\e)(1-\eta)\frac{\psi(x)}{x}\,.
\end{equation}
In particular, by monotonicity of $\psi$ the second term is positive, hence
\begin{equation*}
\psi'(x) \leq -\frac{(1-\e)(1-\eta)}{b}\frac{\psi(x)}{x}
\end{equation*}
in the weak sense in $(x_0,\infty)$; then by testing this inequality with functions approximating $\chi_{(x-z,x)}$ we obtain
\begin{equation} \label{dec3}
\psi(x) \leq \psi(x-z) - \frac{(1-\e)(1-\eta)}{b} \frac{\psi(x)}{x}z
\qquad\text{for all }x>x_0+1 \text{ and }z\in[0,1].
\end{equation}
Plugging \eqref{dec3} into \eqref{dec2} and recalling \eqref{mass2} we have
\begin{equation*}
\psi'(x) \leq - \biggl( \frac{(1-\e)(1-\eta)}{b} + \frac{(1-\e)^3(1-\eta)^2}{b^2} \biggr) \frac{\psi(x)}{x}
\qquad\text{weakly in }(x_0+1,\infty).
\end{equation*}
By iterating the previous argument, we obtain for every $n\in\N$
\begin{equation} \label{dec4}
\psi'(x) \leq -\sum_{k=1}^n \frac{(1-\e)^{2k-1}(1-\eta)^k}{b^k} \frac{\psi(x)}{x}
\qquad\text{weakly in }(x_0+(n-1),\infty).
\end{equation}
Given $\gamma>0$, we can now choose $\e_\gamma>0$, $b_\gamma>1$, $\eta_\gamma>0$ and $n_\gamma\in\N$, depending only on $\gamma$, such that
\begin{equation*}
\sum_{k=1}^{n_\gamma} \frac{(1-\e_\gamma)^{2k-1}(1-\eta_\gamma)^k}{b^k} > \gamma
\qquad\text{for every }b\in[1,b_\gamma].
\end{equation*}
Setting also $x_\gamma=x_0+(n_\gamma-1)$ (notice that also the point $x_\gamma$ depends only on $\gamma$), we then have by \eqref{dec4}
\begin{equation*}
\psi'(x) \leq -\gamma \frac{\psi(x)}{x}
\qquad\text{weakly in }(x_\gamma,\infty)\,.
\end{equation*}
This yields for every $\theta\in C_{\mathrm c}(x_\gamma,\infty)$ with $\theta\geq 0$
\begin{equation*}
\int_{(x_\gamma,\infty)} \bigl( x^\gamma\psi(x) \bigr)'\theta(x)\de x \leq 0\,,
\end{equation*}
which implies $x^\gamma\psi(x)\leq (x_\gamma)^\gamma\psi(x_\gamma)\leq (x_\gamma)^\gamma$ for every $x>x_\gamma$.
Since $\psi(x)\leq1$, the inequality $\psi(x)\leq (x_\gamma/x)^\gamma$ holds actually for every $x\geq0$.
Hence the conclusion of the lemma follows by taking $C_\gamma =(x_\gamma)^\gamma$.
\end{proof}

We have a corresponding dual property for the decay of the mass of $g$ in intervals close to the origin.

\begin{lemma} \label{lem:decaybis}
Given any fixed $\gamma<1$, there exist constants $\tilde{\eta}_\gamma>0$, $\tilde{b}_\gamma>1$ and $\widetilde{C}_\gamma>0$ such that for every $b\in[1,\tilde{b}_\gamma]$, $R_0>0$ and for every $g\in\mathcal{S}_b$, if
\begin{equation*}
\int_{(R_0,\infty)} g(x)\de x >1-\tilde{\eta}_\gamma
\end{equation*}
then for every $x>0$
\begin{equation*}
\int_{(0,x)} g(y)\de y \leq \widetilde{C}_\gamma \Bigl(\frac{x}{R_0}\Bigr)^\gamma\,.
\end{equation*}
\end{lemma}

\begin{proof}
The proof is similar to the one of Lemma~\ref{lem:decay}.
By Remark~\ref{rm:scaling} we assume $R_0=1$ without loss of generality.
We fix $\eta>0$, $b>1$ and $\e>0$, to be chosen later, and let $\sigma=\sigma(\e)>0$ be such that $|K(\xi,1)-1|<\e$ if $\xi<\sigma$.
We also fix $x_0\leq \sigma$.
Set $\psi(x):=\int_{[0,x)}g(y)\de y$, and recall that by Remark~\ref{rm:delta} we have $\psi(0)=0$. By testing the equation \eqref{eqtnmeas} with a function $\theta\in C_{\mathrm c}([0,x_0))$, $\theta\geq0$, we have
\begin{align*}
b\inte x\theta(x)\psi'(x) \de x
& \geq \int_{[0,x_0)} \de y \int_{[1,\infty)} \de z \, K\Bigl(\frac{y}{z},1\Bigr) g(y)g(z)\int_{y}^{y+z}\theta(x)\de x \\
& \geq (1-\e) \int_{[0,x_0)}\de y \int_{[1,\infty]}\de z \, g(y)g(z)\int_{y}^{y+z}\theta(x)\de x \\
& = (1-\e) \int_{[0,x_0)} \de x \int_{[1,\infty]}\de z \, \theta(x)g(z) \int_{[0,x)}g(y)\de y \\
& \geq (1-\e)(1-\eta) \int_{[0,x_0)} \theta(x)\psi(x)\de x\,.
\end{align*}
Given $\gamma<1$, we can now choose $\e_\gamma>0$, $\tilde{\eta}_\gamma>0$ and $\tilde{b}_\gamma>1$ such that for $b\in[1,\tilde{b}_\gamma]$ we have $\gamma<\frac{(1-\e_\gamma)(1-\tilde{\eta}_\gamma)}{b}$. With this choice we hence obtain
\begin{equation*}
\psi'(x) \geq \gamma \frac{\psi(x)}{x}
\qquad\text{weakly in }(0,x_0).
\end{equation*}
This yields for every $\theta\in C_{\mathrm c}((0,x_0))$ with $\theta\geq 0$
\begin{equation*}
\int_{(0,x_0)} \bigl( x^{-\gamma}\psi(x) \bigr)'\theta(x)\de x \geq0\,,
\end{equation*}
which implies $x^{-\gamma}\psi(x)\leq x_0^{-\gamma}\psi(x_0)\leq x_0^{-\gamma}$ for every $x<x_0$.
Since $\psi(x)\leq1$, the inequality $\psi(x)\leq (x/x_0)^\gamma$ holds actually for every $x\geq0$.
Hence the conclusion of the lemma follows recalling that the choice of $x_0$ depends only on $\gamma$, and taking $\widetilde{C}_\gamma =x_0^{-\gamma}$.
\end{proof}

Lemma~\ref{lem:decay} and Lemma~\ref{lem:decaybis} have some useful consequences: firstly, one obtains that the $\alpha$-moment of any measure $g\in\mathcal{S}_b$ is finite (for $b$ sufficiently close to 1), where $\alpha$ is as in assumption \eqref{kernel4}; in view of Remark~\ref{rm:scaling} we can hence normalize the solution in order to have the $\alpha$-moment of $g$ equal to 1. Furthermore, we can show that $g$ is in fact absolutely continuous with respect to the Lebesgue measure, with density in $L^1(0,\infty)$. This is the content of the following proposition.

\begin{proposition} \label{prop:moment}
Let $\alpha$ be as in assumption \eqref{kernel4} and let
\begin{equation*}
\mathcal{S}_b^\alpha := \Bigl\{ g\in\mathcal{S}_b \,:\, g\in L^1(0,\infty),\, M_\alpha:=\inte x^\alpha g(x)\de x =1 \Bigr\}\,.
\end{equation*}
There exists $\bar{b}>1$ such that, for every $b\in[1,\bar{b}]$, if $\mathcal{S}_b\neq\emptyset$ then $\mathcal{S}_b^\alpha\neq\emptyset$.
\end{proposition}

\begin{proof}
Fix $\gamma_1>1$ and let $\eta_{\gamma_1}>0$, $b_{\gamma_1}>1$ and $C_{\gamma_1}>0$ be the constants given by Lemma~\ref{lem:decay}. For $b\in[1,b_{\gamma_1}]$, if $g\in\mathcal{S}_b$ we have
\begin{equation*}
\int_{[0,R_0]}g(x)\de x \geq 1-\eta_{\gamma_1}
\end{equation*}
for some $R_0>0$, since the mass of $g$ is 1. Then by Lemma~\ref{lem:decay}
\begin{equation*}
\psi(x):=\int_{[x,\infty)} g(y)\de y \leq C_{\gamma_1}\Bigl(\frac{R_0}{x}\Bigr)^{\gamma_1}
\end{equation*}
for all $x>0$. Such decay implies that the $\alpha$-moment of $g$ is finite: indeed
\begin{align}\label{dec6}
\inte x^\alpha g(x)\de x &= \alpha\inte x^{\alpha-1}\psi(x)\de x \nonumber\\
& \leq \alpha\int_{[0,1]}x^{\alpha-1}\de x + \alpha C_{\gamma_1} R_0^{\gamma_1} \int_{(1,\infty)} x^{\alpha-1-\gamma_1}\de x <\infty\,.
\end{align}
Similarly, by the same argument we also have
\begin{equation} \label{dec7}
\inte xg(x)\de x <\infty\,.
\end{equation}

Next, fix $\gamma_2\in(1-\alpha,1)$ and let $\tilde{\eta}_{\gamma_2}>0$, $\tilde{b}_{\gamma_2}>1$ and $\widetilde{C}_{\gamma_2}>0$ be given by Lemma~\ref{lem:decaybis}. For $b\in[1,\tilde{b}_{\gamma_2}]$, if $g\in\mathcal{S}_b$ we have
\begin{equation*}
\int_{[R_0,\infty)}g(x)\de x \geq 1-\tilde{\eta}_{\gamma_2}
\end{equation*}
for some $R_0>0$, since the mass of $g$ is 1 and no part of $g$ is concentrated at the origin by Remark~\ref{rm:delta}. Then by Lemma~\ref{lem:decaybis}
\begin{equation*}
\psi(x):=\int_{[0,x)} g(y)\de y \leq \widetilde{C}_{\gamma_2}\Bigl(\frac{x}{R_0}\Bigr)^{\gamma_2}
\end{equation*}
for all $x>0$, which yields integrating by parts
\begin{align} \label{dec5}
\inte \frac{g(x)}{x^{1-\alpha}}\de x &= (1-\alpha)\inte \frac{\psi(x)}{x^{2-\alpha}}\de x \nonumber\\
& \leq (1-\alpha) \widetilde{C}_{\gamma_2}R_0^{-\gamma_2} \int_{[0,1]} \frac{1}{x^{2-\alpha-\gamma_2}} \de x + (1-\alpha)\int_{(1,\infty)} \frac{1}{x^{2-\alpha}}\de x <\infty\,.
\end{align}

Let now $\bar{b}:=\min\{b_{\gamma_1}, \tilde{b}_{\gamma_2}\}>1$, and let $g\in\mathcal{S}_b$, for some $b\in[1,\bar{b}]$. In view of \eqref{dec6} and of Remark~\ref{rm:scaling} we can rescale $g$ and obtain a measure (that for simplicity we still denote by $g$) whose $\alpha$-moment is equal to 1. Notice that the properties \eqref{dec6}--\eqref{dec5} are preserved by the rescaling, as well as the mass of $g$. In order to conclude that this rescaled measure is an element of $\mathcal{S}_{b}^\alpha$, we need to prove that it is absolutely continuous with respect to the Lebesgue measure, with density in $L^1(0,\infty)$.

Let $p=\frac{1}{\alpha}$ and let $p'=\frac{p}{p-1}$ be its conjugate exponent. We then have for every $\theta\in C_{\mathrm{c}}([0,\infty))$, using H\"older's inequality in the weak formulation \eqref{eqtnmeas} of the equation,
\begin{align*}
b\inte x\theta(x)g(x)\de x
& \leq \|\theta\|_{L^{p'}} \inte\inte \frac{K(y,z)}{z^{1-\alpha}}g(y)g(z)\de y \de z \\
& = \|\theta\|_{L^{p'}} \inte \int_{[0,y]} \frac{yK(1,\frac{z}{y})}{z^{1-\alpha}}g(y)g(z) \de z \de y\\
& \qquad + \|\theta\|_{L^{p'}} \inte \int_{(y,\infty)} z^\alpha K\Bigl(\frac{y}{z},1\Bigr)g(y)g(z) \de z \de y
\leq  C \|\theta\|_{L^{p'}}
\end{align*}
for some constant $C>0$, where the last estimate follows easily from the continuity of $K$ and from \eqref{dec6}--\eqref{dec5}.
This yields $xg(x)\in L^p(0,\infty)$ and, in turn, $g\in L^1(0,\infty)$, as claimed.
\end{proof}

By Proposition~\ref{prop:moment}, the proof of Theorem~\ref{thm:nonexist} amounts to show that $\mathcal{S}_b^\alpha=\emptyset$ for $b$ sufficiently close to 1.
We conclude this preliminary subsection by stating two additional technical lemmas, dealing with further decay properties of the elements of $\mathcal{S}^\alpha_b$, that will be needed in the following.

\begin{lemma} \label{lem:moment}
Let $\alpha$ be as in assumption \eqref{kernel4}. For every $b\in[1,\bar{b}]$ and for every $g\in\mathcal{S}_b^\alpha$ one has the estimates
\begin{align*}
\int_{[x,\infty)}g(z) & \de z \leq C_\alpha x^{-(\alpha+2)},
\qquad
\int_{[x,\infty)} z^\alpha g(z)\de z \leq C_\alpha x^{-2}, \\
& \inte zg(z)\de z + \inte z^{1+\alpha}g(z)\de z \leq C_\alpha\,,
\end{align*}
for a constant $C_\alpha>0$ depending only on $\alpha$.
\end{lemma}

\begin{proof}
Let $\gamma=\alpha+2$ and let $\eta_\gamma>0$, $C_\gamma>0$ and $b_\gamma>1$ be given by Lemma~\ref{lem:decay}.
By taking a smaller $\bar{b}$ if necessary, we can assume $\bar{b}\leq b_\gamma$. 
If $g\in\mathcal{S}_b^\alpha$ for some $b\in[1,\bar{b}]$, then we have, by the fact that the $\alpha$-moment of $g$ is normalized to 1,
\begin{align*}
\int_{[0,R_0]}g(x)\de x
= 1 - \int_{(R_0,\infty)}g(x)\de x
\geq 1 - \frac{1}{R_0^\alpha} \inte x^\alpha g(x)\de x
\geq 1-\frac{1}{R_0^\alpha} \geq 1-\eta_\gamma
\end{align*}
if $R_0$ is sufficiently large (depending ultimately only on $\alpha$). By Lemma~\ref{lem:decay} we then have
\begin{equation*}
\int_{[x,\infty)}g(z)\de z \leq C_{\gamma}\biggl(\frac{R_0}{x}\biggr)^{\gamma} = C_\alpha x^{-(\alpha+2)}\,.
\end{equation*}
The estimates in the statement follow now easily from this bound.
\end{proof}

\begin{lemma} \label{lem:final}
Let $g\in\mathcal{S}_b^\alpha$ for some $b\in[1,\bar{b}]$. Then for every $R>0$ one has
\begin{equation} \label{final}
\int_{(R,\infty)}g(x)\de x \leq \frac{1}{R^{\alpha}}\,.
\end{equation}
Furthermore, for every $D>0$ there exists a constant $C(D)>0$, depending only on $D$, such that
\begin{equation*}
\int_{(D,\infty)}g(x) \de x \geq C(D)\,.
\end{equation*}
\end{lemma}

\begin{proof}
For every $R>0$ we have
\begin{equation*}
\int_{(R,\infty)} g(x)\de x \leq \frac{1}{R^\alpha}\int_{(R,\infty)}x^\alpha g(x)\de x \leq \frac{1}{R^\alpha}\inte x^\alpha g(x)\de x =\frac{1}{R^\alpha}\,,
\end{equation*}
which is \eqref{final}.
We next prove the second claim by contradiction. Assume that for some $D>0$ there exist sequences $\{b_k\}_k\subset[1,\bar{b}]$ and $\{g_k\}_k\subset\mathcal{S}_{b_k}^\alpha$ such that $\int_{(D,\infty)}g_k(x)\de x \leq \frac{1}{k}$. Since the measures $g_k$ satisfy uniformly the tightness condition \eqref{final}, by the bound on the mass we can extract a (not relabeled) subsequence such that $b_k\to b\in[1,\bar{b}]$ and
$$
\inte\theta(x)g_k(x)\de x \to \inte \theta\de \mu
\qquad\text{for every } \theta\in C([0,\infty)), \theta\text{ bounded,}
$$
for some measure $\mu\in\meas$ with $\inte\de \mu=1$.

The condition $\inte x^\alpha g_k(x)\de x=1$ excludes that $\mu$ is concentrated at the origin: indeed, by Lemma~\ref{lem:moment} we can find $R>0$, independent of $k$, such that
$$
\int_{[0,R]}x^{\alpha}g_k(x)\de x = 1 - \int_{(R,\infty)} x^\alpha g_k(x)\de x \geq \frac12\,.
$$
If $\mu=\delta_0$, then by taking a nonnegative cut-off function $\theta\in C_{\mathrm c}([0,R+1))$ with $\theta\equiv1$ in $[0,R]$ we would get
\begin{align*}
0 = \inte x^\alpha\theta(x)\de\mu
= \lim_{k\to\infty}\inte x^\alpha\theta(x) g_k(x)\de x
\geq \int_{[0,R]} x^\alpha g_k(x)\de x \geq\frac12\,,
\end{align*}
which would be a contradiction. Hence $\mu\neq\delta_0$. By passing to the limit as $k\to\infty$ in the weak formulation \eqref{eqtnmeas} of the equation solved by $g_k$, we hence obtain that $\mu\in\mathcal{S}_b$.

On the other hand, we have $\int_{(D,\infty)}\de \mu=0$, hence $\mu$ has compact support. We claim that this is not possible, as $\mu$ is a weak solution to \eqref{selfsim4bis}. To see this, let $M:=\max\{x:x\in\supp\mu\}$, and fix $\e>0$ such that $\mu([M-\e,M])>0$. By taking a nonnegative test function $\theta\in C_{\mathrm{c}}([0,\infty))$ with $\supp\theta\cap\supp\mu=\emptyset$, $\theta\equiv1$ in $[M+\e,2M]$, we have by \eqref{eqtnmeas}
\begin{align*}
0 &= \inte\inte \frac{K(y,z)}{z}\int_y^{y+z}\theta(x)\de x \de\mu(y)\de \mu(z) \\
& \xupref{kernel5}{\geq} k_0 \int_{[M-\e,M]}\int_{[M-\e,M]}\int_{y}^{y+z}\theta(x)\de x \de\mu(y)\de\mu(z) \\
& \geq k_0 \int_{[M-\e,M]}\int_{[M-\e,M]} \int_{M+\e}^{2M-2\e}\theta(x)\de x \de\mu(y)\de\mu(z) \\
& \geq k_0(M-3\e) \bigl(\mu([M-\e,M])\bigr)^2>0\,,
\end{align*}
which is a contradiction.
\end{proof}


\subsection{The duality argument}
The core of the contradiction argument leading to Theorem~\ref{thm:nonexist} relies on the observation that, if $g$ is a weak solution to \eqref{selfsim4bis} in the sense of Definition~\ref{def:weaksol}, then for every $T>0$ and for every test function $\vphi\in C^1_{\mathrm c}([0,\infty)\times[0,T])$ we have
\begin{align} \label{evolution}
\inte g(x)&\vphi(x,T)\de x - \inte g(x)\vphi(x,0)\de x \nonumber\\
& - \int_0^T \inte \partial_t\vphi(x,t)g(x)\de x \de t
- b \int_0^T\inte x \partial_x\vphi(x,t)g(x)\de x \de t \nonumber\\
& + \int_0^T\inte\inte \frac{K(y,z)}{z}g(y)g(z)\bigl[\vphi(y+z,t)-\vphi(y,t)\bigr] \de y \de z \de t = 0\,.
\end{align}
By choosing $\vphi$ as a subsolution to the adjoint equation
\begin{equation} \label{adjoint}
\mathcal{L}_b(\vphi(x,t)):=\partial_t \vphi(x,t) + bx\partial_x\vphi(x,t) - \inte \frac{K(x,z)}{z}g(z)\bigl[\vphi(x+z,t)-\vphi(x,t)\bigr]\de z = 0\,,
\end{equation}
then
\begin{equation} \label{duality}
\inte \vphi(x,T)g(x)\de x \leq \inte \vphi(x,0)g(x)\de x\,.
\end{equation}

In the following lemma we construct an explicit subsolution $\phibar$ for the operator $\mathcal{L}_1$, for which \eqref{duality} holds with an additional remainder on the right-hand side, which can be made uniformly small for $b$ close to 1. At the initial time, the function $\phibar(x,0)$ is going to be identically equal to one in regions far away from the origin, that is for $x\in(R_0,\infty)$, so that the right-hand side of \eqref{duality} tends to 0 as $R_0\to\infty$ and $b\to1$.

On the other hand, as explained in the Introduction, the main effect of the adjoint equation \eqref{adjoint} is to transport the mass towards the origin: we show that after some positive time $T$ the function $\phibar(x,T)$ becomes uniformly positive in regions of order one, so that the left-hand side of \eqref{duality} is bounded from below (up to a uniform constant) by the integral of $g$ itself in an interval of the form $(D,\infty)$, for some uniform constant $D$. Actually, we cannot reach a region $(D,\infty)$, with $D$ independent of $R_0$, in a single step: we instead perform an iteration argument with a sequence of subsolutions, taking at each step an initial datum close to the subsolution of the previous step at the final time $T$; by this procedure, in a finite number of step we eventually arrive in a region of order one.

\begin{lemma} \label{lem:subsolution}
There exist constants $D>0$ and $M>0$, depending only on the kernel $K$, with the following property.
Let $b\in(1,\bar{b})$, where $\bar{b}$ is given by Proposition~\ref{prop:moment}, and let $g\in\mathcal{S}_b^\alpha$.
For every $R_0>D$ one has
\begin{equation*}
\int_{(D,\infty)}g(x)\de x \leq M \int_{(R_0,\infty)}g(x)\de x + (b-1)\omega(R_0)\,,
\end{equation*}
where $\omega(R_0)$ is a constant depending only on $R_0$ and on the kernel $K$.
\end{lemma}

\begin{proof}
Fix any $g\in\mathcal{S}_b^\alpha$ for $b\in(1,\bar{b})$. We start by constructing a suitable subsolution for the operator $\mathcal{L}_1$ corresponding to this $g$, defined in \eqref{adjoint}.

Fix a positive constant $A>0$, to be chosen later. For a given $R>A$, let $\phibar(x,t)$ be the solution to
\begin{align}
\begin{cases} \label{subsol1}
\partial_t\phibar - \frac12\beta_0x^{1-\alpha}\partial_x\phibar=0\,,\\
\phibar(x,0)=\vphi_0\bigl(\frac{x}{R}\bigr)\,,
\end{cases}
\end{align}
where the initial profile $\vphi_0$ is defined by $\vphi_0(x):=\bigl(1-\frac{1}{x}\bigr)^+$, $(\cdot)^+$ denoting the positive part. Here $\alpha\in(0,1)$ and $\beta_0>0$ are the constants appearing in assumption \eqref{kernel4}.
Notice that the characteristics of \eqref{subsol1} are the curves
\begin{equation} \label{subsol2}
x(t;z)= \bigl( z^\alpha - \Gamma t \bigr)^{\frac{1}{\alpha}}\,,
\qquad \Gamma:= \frac{\alpha\beta_0}{2}\,,
\end{equation}
so that the function $\phibar$ is explicitly given by
\begin{equation} \label{subsol3}
\phibar(x,t)= \vphi_0\Bigl(\frac{z(x,t)}{R}\Bigr)\,,
\qquad z(x,t):= \bigl( x^\alpha + \Gamma t \bigr)^{\frac{1}{\alpha}}\,.
\end{equation}
We finally let $T=T(R,A)$ be the time such that $x(T;R)=A$, which is given by
\begin{equation} \label{subsol5}
T= \frac{R^\alpha-A^\alpha}{\Gamma}\,.
\end{equation}
We now divide the rest of the proof into three steps.

\medskip\noindent
\textit{Step 1: verification that $\phibar$ is a subsolution for $\mathcal{L}_1$.} We claim that, if $A$ is sufficiently large (depending only on the kernel $K$), then for every $R>A$ the function $\phibar$ defined in \eqref{subsol3} satisfies
\begin{equation} \label{subsol6}
\mathcal{L}_1(\phibar(x,t))\leq 0 \qquad\text{for all }t\in[0,T] \text{ and }x\in[0,\infty)\setminus\{x(t;R)\},
\end{equation}
where the time $T$ is defined by \eqref{subsol5}.

It is clearly sufficient to check \eqref{subsol6} for $x>x(t;R)$.
By assumption \eqref{kernel4} we can find $\delta>0$ such that $K(1,\xi)\geq 1 +\frac34\beta_0\xi^\alpha$ for all $\xi\in[0,\delta]$. We then have, using the homogeneity of $K$ and the monotonicity of $\phibar$ with respect to $x$,
\begin{align*}
& \inte \frac{K(x,y)}{y}g(y)\bigl[\phibar(x+y,t)-\phibar(x,t)\bigr]\de y
\geq \int_{[0,\delta x]} K\Bigl(1,\frac{y}{x}\Bigr) \frac{x}{y} g(y)\bigl[\phibar(x+y,t)-\phibar(x,t)\bigr]\de y \nonumber\\
& \geq x\int_{[0,\delta x]}\frac{g(y)}{y} \bigl[\phibar(x+y,t)-\phibar(x,t)\bigr] \de y 
+\frac34\beta_0x^{1-\alpha} \int_{[0,\delta x]}\frac{g(y)}{y^{1-\alpha}} \bigl[\phibar(x+y,t)-\phibar(x,t)\bigr] \de y \nonumber\\
& \geq x\partial_x\phibar(x,t)\int_{[0,\delta x]}g(y)\de y + \frac34\beta_0x^{1-\alpha}\partial_x\phibar(x,t)\int_{[0,\delta x]}y^\alpha g(y)\de y \nonumber\\
& \qquad - \biggl( x\int_{[0,\delta x]}yg(y)\de y + \frac34\beta_0x^{1-\alpha}\int_{[0,\delta x]} y^{1+\alpha}g(y)\de y\biggr) \sup_{\xi\in[x,x+\delta x]} |\partial^2_{xx}\phibar(\xi,t)|\,.
\end{align*}
By using this estimate and \eqref{subsol1}, and recalling that both the mass and the $\alpha$-moment of $g$ are normalized to 1, we have for every $x>x(t;R)$
\begin{align*}
\mathcal{L}_1(\phibar(x,t)) &=\frac{\beta_0}{2}x^{1-\alpha}\partial_x\phibar + x\partial_x\phibar - \inte\frac{K(x,y)}{y}g(y)\bigl[\phibar(x+y,t)-\phibar(x,t)\bigr]\de y \nonumber\\
& \leq -\frac{\beta_0}{4}x^{1-\alpha}\partial_x\phibar + x\partial_x\phibar\int_{(\delta x,\infty)}g(y)\de y + \frac34\beta_0x^{1-\alpha}\partial_x\phibar\int_{(\delta x,\infty)}y^\alpha g(y)\de y \nonumber\\
& \qquad + \biggl( x\int_{[0,\delta x]}yg(y)\de y + \frac34\beta_0x^{1-\alpha}\int_{[0,\delta x]} y^{1+\alpha}g(y)\de y\biggr) \sup_{\xi\in[x,x+\delta x]} |\partial^2_{xx}\phibar(\xi,t)|\,.
\end{align*}
Using Lemma~\ref{lem:moment} we easily obtain
\begin{align*}
\mathcal{L}_1(\phibar(x,t))
\leq -\frac{\beta_0}{4}x^{1-\alpha}\partial_x\phibar + C x^{-(1+\alpha)}\partial_x\phibar + C\bigl(x+x^{1-\alpha}\bigr)\sup_{\xi\in[x,x+\delta x]} |\partial^2_{xx}\phibar(\xi,t)|
\end{align*}
for a constant $C$ depending only on $\alpha$, $\beta_0$ and $\delta$.
Notice now that $x>x(t;R)\geq A$ for $t\in[0,T]$, hence by choosing $A$ large enough (also depending only on $\alpha$, $\beta_0$ and $\delta$) we deduce
\begin{align*}
\mathcal{L}_1(\phibar(x,t))
& \leq -\frac{\beta_0}{8}x^{1-\alpha}\partial_x\phibar(x,t) + Cx\sup_{\xi\in[x,x+\delta x]} |\partial^2_{xx}\phibar(\xi,t)| \nonumber\\
& \leq -\frac{\beta_0}{8}x^{1-\alpha}\partial_x\phibar(x,t) + Cx|\partial^2_{xx}\phibar(x,t)|\,,
\end{align*}
where the last inequality follows straightforwardly from the explicit expression of $\phibar$ (possibly taking a larger constant $C$, depending on the same parameters). By \eqref{subsol3} it follows
\begin{align} \label{subsol7}
\mathcal{L}_1(\phibar(x,t)) \leq -\frac{\beta_0}{8R}x^{1-\alpha}\vphi_0'({\textstyle\frac{z}{R}})\partial_x z + \frac{Cx}{R^2}|\vphi_0''({\textstyle\frac{z}{R}})| (\partial_x z)^2 + \frac{Cx}{R}\vphi_0'({\textstyle\frac{z}{R}})|\partial^2_{xx}z|
=: - l_1 + l_2 + l_3\,,
\end{align}
where $z=z(x,t)$ is defined in \eqref{subsol3}.
We now use the explicit expressions of $\vphi_0'$, $\vphi_0''$ and
\begin{align*}
\partial_x z = \Bigl(\frac{z}{x}\Bigr)^{1-\alpha},
\qquad
\partial^2_{xx}z = (1-\alpha)\Bigl( x^{\alpha-1}z^{-\alpha}\partial_x z - x^{\alpha-2}z^{1-\alpha} \Bigr)
\end{align*}
to show that the right-hand side of \eqref{subsol7} is negative. Indeed, it is straightforward to check that
\begin{align*}
\frac{l_2}{l_1}
= Cx^{2\alpha-1}z^{-\alpha}
= \frac{Cx^{2\alpha-1}}{x^\alpha + \Gamma t}
\leq \frac{C}{x^{1-\alpha}}
\leq \frac{C}{A^{1-\alpha}}\,,
\end{align*}
and similarly
\begin{align*}
\frac{l_3}{l_1} = C x^\alpha \frac{|\partial^2_{xx}z|}{\partial_xz} 
\leq Cx^{2\alpha-1}z^{-\alpha} + \frac{C}{x^{1-\alpha}} \leq \frac{C}{A^{1-\alpha}}\,.
\end{align*}
By plugging these estimates in \eqref{subsol7} it follows that for $A$ large enough
\begin{equation*}
\mathcal{L}_1(\phibar(x,t)) \leq -\frac{l_1}{2} = -\frac{\beta_0}{16}x^{1-\alpha}\partial_x\phibar(x,t)\leq 0
\end{equation*}
for all $t\in[0,T]$ and $x>x(t;R)$. Notice in particular that, as $C$ depends only on $\alpha$, $\beta_0$ and $\delta$, the constant $A$ depends ultimately only on the kernel $K$. Hence the claim \eqref{subsol6} is proved. We remark that the previous argument strongly relies on the assumption that $\beta_0>0$.

Since $\phibar$ is a subsolution for the operator $\mathcal{L}_1$, by using it as test function in \eqref{evolution} (notice that this is possible even if $\partial_x\phibar$ is discontinuous at one point, since $g\in L^1$) we conclude that
\begin{equation} \label{subsol8}
\int_{[0,\infty)}\phibar(x,T)g(x)\de x \leq \int_{[0,\infty)}\phibar(x,0)g(x)\de x + (b-1)\int_0^T\inte x\partial_x\phibar(x,t)g(x)\de x\de t\,.
\end{equation}

\medskip\noindent
\textit{Step 2: iteration argument.}
We fix two decreasing sequences $(\delta_n)_n$, $(\e_n)_n$ of positive real numbers with the properties
\begin{equation} \label{subsol9}
\lim_{n\to\infty}\delta_n=\lim_{n\to\infty}\e_n=0,\qquad
\frac{\delta_n}{\delta_{n-1}}\geq\frac12\,,\qquad
\prod_{k=1}^\infty (1-\e_k) = q >0\,,
\end{equation}
with $\delta_0=1$, where $q$ is a uniform constant. Let $R_0>2^{\frac{1+\alpha}{\alpha}}A$, where $A$ is the constant determined in the previous step, and $R_k:=\delta_k R_0$. We select $\bar{n}$ as the first integer such that $R_{\bar{n}}\leq 2^{\frac{1+\alpha}{\alpha}}A$; notice that by the second condition in \eqref{subsol9} we have
\begin{equation} \label{subsol9bis}
R_k \geq R_{\bar{n}} \geq 2^{\frac{1}{\alpha}} A
\qquad\text{for all }k=0,1,\ldots,\bar{n}.
\end{equation}
We then consider the functions
\begin{equation} \label{subsol10}
\phibar_k(x,t):=\vphi_0\Bigl(\frac{z(x,t)}{R_k}\Bigr)\,,
\qquad k=0,1,\ldots\bar{n},
\end{equation}
where $z(x,t)$ is the map defined in \eqref{subsol3}. Each function $\phibar_k$ is the solution to \eqref{subsol1} with initial datum $\vphi_0(\frac{x}{R_k})$. By the first step of the proof, and in particular by \eqref{subsol8}, we hence obtain that for every $k=0,\ldots,\bar{n}$
\begin{equation} \label{subsol11}
\inte \phibar_k(x,T_k) g(x) \de x \leq \inte \phibar_k(x,0) g(x) \de x
+ (b-1)\int_0^{T_k}\inte x\partial_x\phibar_k(x,t)g(x)\de x\de t\,,
\end{equation}
where $T_k$ is defined by \eqref{subsol5} (with $R$ replaced by $R_k$). We now claim that the two sequences $(\delta_n)_n$, $(\e_n)_n$ can be chosen such that
\begin{equation} \label{subsol12}
(1-\e_k)\phibar_k(x,0) \leq \phibar_{k-1}(x,T_{k-1})
\end{equation}
for all $k=1,\ldots,\bar{n}$.
Assuming momentarily that the claim is proved, by combining \eqref{subsol11} and \eqref{subsol12} and iterating we get
\begin{align} \label{subsol12bis}
\inte \phibar_{\bar{n}}(x,0) & g(x)\de x
\leq \frac{1}{\prod_{k=1}^{\bar{n}}(1-\e_k)} \inte\phibar_0(x,0)g(x)\de x \nonumber \\
& \qquad\qquad + (b-1)\sum_{k=1}^{\bar{n}}\frac{1}{\prod_{j=k}^{\bar{n}}(1-\e_j)} \int_0^{T_{k-1}}\inte x\partial_x\phibar_{k-1}(x,t)g(x)\de x \de t \nonumber \\
& \leq \frac{1}{q}\int_{(R_0,\infty)} g(x)\de x  + \frac{b-1}{q} \sum_{k=1}^{\bar{n}}\int_0^{T_{k-1}}\inte x\partial_x\phibar_{k-1}(x,t)g(x)\de x \de t\,,
\end{align}
where the last inequality follows from \eqref{subsol9} and from $\phibar_0(x,0)=\vphi_0(\frac{x}{R_0})\leq\chi_{(R_0,\infty)}(x)$.
To estimate the last term in\ \eqref{subsol12bis}, we observe that the partial derivative $\partial_x\phibar_{k-1}(x,t)$ vanishes for $x<x(t;R_{k-1})$, while for $x>x(t;R_{k-1})\geq A$ we have
\begin{equation*}
\partial_x\phibar_{k-1}(x,t) = \frac{R_{k-1}x^{\alpha-1}}{(x^\alpha+\Gamma t)^{\frac{1+\alpha}{\alpha}}} \leq \frac{R_{k-1}x^{\alpha-1}}{A^{1+\alpha}}\,.
\end{equation*}
Hence the last term in \eqref{subsol12bis} can be bounded by
\begin{align*}
\frac{b-1}{q} \sum_{k=1}^{\bar{n}}\int_0^{T_{k-1}}\inte x\partial_x\phibar_{k-1}(x,t)&g(x)\de x \de t
\leq \frac{b-1}{qA^{1+\alpha}} \sum_{k=1}^{\bar{n}} R_{k-1}T_{k-1}\inte x^{\alpha}g(x)\de x \\
& \leq \frac{(b-1)R_0^{1+\alpha}}{q\Gamma A^{1+\alpha}} \sum_{k=1}^{\bar{n}} (\delta_{k-1})^{1+\alpha}\,,
\end{align*}
where in the second inequality we used the fact that the $\alpha$-moment of $g$ is normalized to 1, and the definition of $T_k$ given by \eqref{subsol5}. Going back to \eqref{subsol12bis} we obtain
\begin{align} \label{subsol12ter}
\inte \phibar_{\bar{n}}(x,0) g(x)\de x
\leq \frac{1}{q}\int_{(R_0,\infty)} g(x)\de x 
+ \frac{(b-1)R_0^{1+\alpha}}{q\Gamma A^{1+\alpha}}\sum_{k=1}^{\bar{n}} (\delta_{k-1})^{1+\alpha}\,.
\end{align}
On the left-hand side, we have $\phibar_{\bar{n}}(x,0)=\vphi_0(\frac{x}{R_{\bar{n}}})\geq\frac12$ if $x\geq 2^{\frac{2\alpha+1}{\alpha}}A$, since $R_{\bar{n}}\leq2^{\frac{1+\alpha}{\alpha}}A$ by the choice of $\bar{n}$.
Hence by \eqref{subsol12ter} we finally deduce that
\begin{equation*}
\frac12\int_{(2^{\frac{2\alpha+1}{\alpha}}A,\infty)}g(x)\de x
\leq \frac{1}{q}\int_{(R_0,\infty)}g(x)\de x + \frac{(b-1)R_0^{1+\alpha}}{q\Gamma A^{1+\alpha}}\sum_{k=1}^{\bar{n}} (\delta_{k-1})^{1+\alpha}\,.
\end{equation*}
The conclusion follows by choosing $D=2^{\frac{2\alpha+1}{\alpha}}A$, $M=\frac{2}{q}$, and $\omega(R_0)=\frac{2R_0^{1+\alpha}}{q\Gamma A^{1+\alpha}}\sum_{k=1}^{\bar{n}}(\delta_{k-1})^{1+\alpha}$ (notice in particular that the choice of $\bar{n}$ depends only on $R_0$, $A$ and $\alpha$, so that $\omega(R_0)$ depends only on $R_0$ and the properties of the kernel, as claimed).

\medskip\noindent
\textit{Step 3: proof of claim \eqref{subsol12}.}
We are then left with the proof that we can find two sequences $(\delta_k)_k$, $(\e_k)_k$ such that \eqref{subsol9} and \eqref{subsol12} are simultaneously satisfied. Recalling \eqref{subsol10} we rewrite \eqref{subsol12} as
\begin{align} \label{subsol13}
(1-\e_k)\Bigl( 1-\frac{\delta_kR_0}{x} \Bigr)^+ \leq \Bigl( 1-\frac{\delta_{k-1}R_0}{z(x,T_{k-1})} \Bigr)^+\,.
\end{align}
Using the definition \eqref{subsol3} of $z(x,t)$ and passing to the variable $y:=\frac{x}{\delta_{k-1}R_0}$ we see that \eqref{subsol13} is equivalent to
\begin{equation} \label{subsol14}
(1-\e_k)\biggl( 1-\frac{\delta_k}{\delta_{k-1}y} \biggr)^+
\leq \biggl( 1 - \frac{1}{(y^\alpha + \Delta_k)^\frac{1}{\alpha}} \biggr)^+\,,
\end{equation}
where $\Delta_k:=\frac{\Gamma T_{k-1}}{(\delta_{k-1})^\alpha R_{0}^\alpha}=1-(\frac{A}{\delta_{k-1}R_0})^\alpha$ by \eqref{subsol5}.
It is obviously sufficient to check the inequality \eqref{subsol14} for $y\geq\frac{\delta_k}{\delta_{k-1}}$.
Recalling \eqref{subsol9bis}, it is easy to see that for $y\geq\frac{\delta_k}{\delta_{k-1}}$ also the right-hand side of \eqref{subsol14} is strictly positive, so that we are left with the proof of the inequality
\begin{equation*}
\frac{1}{(y^\alpha + \Delta_k)^\frac{1}{\alpha}} - \frac{\delta_k(1-\e_k)}{\delta_{k-1}y} -\e_k \leq 0
\end{equation*}
for $y\geq\frac{\delta_k}{\delta_{k-1}}$.
Observing also that $\Delta_k\geq\frac12$ by \eqref{subsol9bis}, it is sufficient to check that
\begin{equation} \label{subsol15}
f_k(y):=\frac{1}{(y^\alpha + \frac12)^\frac{1}{\alpha}} - \frac{\delta_k(1-\e_k)}{\delta_{k-1}y} -\e_k \leq 0\,.
\end{equation}
By computing the derivative, one finds that the maximum of $f_k$ is the value
\begin{equation*}
\bar{f}_k = 2^{\frac{1}{\alpha}}(1-a_k)^{\frac{1+\alpha}{\alpha}} - \e_k\,,
\qquad
a_k := \Bigl( \frac{\delta_k}{\delta_{k-1}}(1-\e_k) \Bigr)^{\frac{\alpha}{\alpha+1}}\,.
\end{equation*}
We then choose the two sequences $\delta_k=\frac{1}{k+1}$, $\e_k := \frac{2^{\frac{1}{\alpha}}}{(k+1)^{\frac{1+\alpha}{\alpha}}}$, which clearly satisfy the conditions in \eqref{subsol9}.
It can be proved that with this choice the inequality
\begin{equation*}
(1-\e_k)^\alpha \geq \frac{\delta_k}{\delta_{k-1}}
\end{equation*}
holds for every $k\geq1$; in turn, we obtain $a_k\geq 1 - \frac{1}{k+1}$, which implies that $\bar{f}_k\leq0$ for every $k\geq1$, which is what we had to prove. This completes the proof of the claim \eqref{subsol12} and, in turn, of the lemma.
\end{proof}

We are now in position to conclude the proof of Theorem~\ref{thm:nonexist}.

\begin{proof}[Proof of Theorem~\ref{thm:nonexist}]
Assume that there exists a solution $g\in\mathcal{S}_b^\alpha$ for some $b\in(1,\bar{b})$, where $\bar{b}$ is given by Proposition~\ref{prop:moment}. Let $D>0$ be the constant given by Lemma~\ref{lem:subsolution}; then by the same lemma we have for every $R_0>D$
\begin{align*}
C(D) \leq \int_{(D,\infty)}g(x)\de x
\leq M\int_{(R_0,\infty)}g(x)\de x + (b-1)\omega(R_0)
\leq \frac{M}{R_0^\alpha} + (b-1)\omega(R_0)\,,
\end{align*}
where we used Lemma~\ref{lem:final} in the first and in the last inequalities.
We obtain a contradiction by taking $R_0$ large enough and, consequently, $b$ sufficiently close to 1; this yields $\mathcal{S}^\alpha_b=\emptyset$ and in turn, by Proposition~\ref{prop:moment}, also $\mathcal{S}_b=\emptyset$, which concludes the proof of the theorem.
\end{proof}


\bigskip
\bigskip
\noindent
{\bf Acknowledgments.}
The authors acknowledge support through the CRC 1060 \textit{The mathematics of emergent effects} at the University of Bonn that is funded through the German Science Foundation (DFG).

\bibliographystyle{siam}
\bibliography{bibliography}

\end{document}